\documentclass[11pt,a4paper,leqno]{amsart}
\usepackage[T1]{fontenc}
\usepackage{amsmath,amssymb,mathtools}
\usepackage{palatino}
\usepackage[margin=30mm]{geometry}
\usepackage[colorlinks=true,linkcolor=blue,citecolor=blue,urlcolor=blue]{hyperref}
\numberwithin{equation}{section}
\swapnumbers
\newtheorem{theorem}{Theorem}[section]
\newtheorem{lemma}[theorem]{Lemma}
\newtheorem{corollary}[theorem]{Corollary}
\theoremstyle{definition}
\newtheorem{definition}[theorem]{Definition}
\theoremstyle{remark}

\newcommand{\R}{\mathbb R}
\newcommand{\Z}{\mathbb Z}
\newcommand{\cD}{\mathcal D}
\newcommand{\1}{\mathbf 1}
\newcommand{\Id}{\mathrm{Id}}
\newcommand{\ch}{\operatorname{ch}}
\newcommand{\sh}{\operatorname{sh}}
\newcommand{\ud}{\,\mathrm d}

\title[The Zygmund and Rey conjectures]
{The Zygmund conjecture and Rey's exponential integrability conjecture}
\author{Henri Martikainen}
\address[H.M.]{Department of Mathematics, Washington University in St. Louis,
1 Brookings Drive, St. Louis, MO 63130, USA}
\email{henri@wustl.edu}
\date{}
\hypersetup{pdftitle={The Zygmund conjecture and Rey's exponential integrability conjecture},
pdfauthor={Henri Martikainen}}

\begin{document}
\begin{abstract}
We prove the Zygmund conjecture in all parameters: the maximal
operator associated with products of cubes with side lengths
$(s_1,\ldots,s_{m-1},\phi(s_1,\ldots,s_{m-1}))$,
where $\phi$ is any positive coordinatewise nondecreasing function,
satisfies the weak $L(\log L)^{m-2}$ estimate for every $m\geq3$.
The estimate is new for $m>3$.
Our approach is to prove Rey's recent conjecture on the exponential
integrability of the overlap function
$h_{\mathcal G}=\sum_{I\in\mathcal G}\1_I$, with exponent $1/(m-1)$,
for sparse incomparable families $\mathcal G$ of $m$-parameter
dyadic rectangles, $m\geq2$.
The Zygmund conjecture, even in the continuous form above, follows
by strengthening a reduction given by Rey.
\end{abstract}
\maketitle
\allowdisplaybreaks[2]

\section{Introduction}

Let $\R^d=\R^{d_1}\times\cdots\times\R^{d_m}$, where
$d=d_1+\cdots+d_m$. For $m\geq3$ and a coordinatewise nondecreasing
function $\phi\colon(0,\infty)^{m-1}\to(0,\infty)$, let
$\mathcal R_\phi$ consist of all rectangles
$R=R^1\times\cdots\times R^m$ with axis-parallel cube factors
and side-length tuple
\[
 (s_1,\ldots,s_{m-1},\phi(s_1,\ldots,s_{m-1})),
 \qquad s_1,\ldots,s_{m-1}>0.
\]
For a rectangle family $\mathcal E$, write
$M_{\mathcal E}f(x)=\sup_{I\in\mathcal E,\,x\in I}
|I|^{-1}\int_I|f|$.
The Zygmund conjecture predicts the weak $L(\log L)^{m-2}$ estimate
\begin{equation}\label{eq:zygmund}
 |\{x:M_{\mathcal R_\phi}f(x)>\lambda\}|
 \leq C_{m,d}\int_{\R^d}\frac{|f(x)|}{\lambda}
 \left[\log\!\left(e+\frac{|f(x)|}{\lambda}\right)\right]^{m-2}\ud x,
 \qquad\lambda>0,
\end{equation}
with a constant independent of $\phi$.
The full $m$-parameter strong maximal function has the logarithmic
power $m-1$: the conjecture says that making one side length depend
monotonically on the others saves one logarithm, reducing this power
to $m-2$.
The origins of Zygmund's conjecture are sometimes dated to circa 1935;
see Mart\'inez~\cite{Martinez}. C\'ordoba's solution~\cite{Cordoba} of the case
$m=3$ in $\R^3$ itself dates back nearly half a century.

Zygmund also proposed a more general conjecture in which
all side lengths are nondecreasing functions of fewer parameters.
Soria~\cite{Soria} disproved that broader conjecture; the formulation
above survives his counterexamples. The dyadic version
(with $d_1=\cdots=d_m=1$, but this is not important) is stated as
Conjecture~1 in Rey~\cite{Rey},
who gave a new proof of the dyadic case $m=3$.
We prove \eqref{eq:zygmund} for
every $m\geq3$ in
Theorem~\ref{thm:continuous-zygmund}. This proof has also been formalized
in Lean~4; see the end of the Introduction.
For a selection of work on Zygmund's conjecture and related
differentiation problems, see
\cite{Cordoba,FeffermanPipher,Gauvan,HagelsteinOnianiStokolos,HagelsteinStokolos,Martinez,Soria,Stokolos,Zygmund}.

Our approach uses Rey's recent work~\cite{Rey} on how sparseness
and incomparability of a rectangle family control its overlap.
We now allow $m\geq2$ and let
$\cD=\cD^1\times\cdots\times\cD^m$ be a product of dyadic
cube grids on $\R^{d_1}\times\cdots\times\R^{d_m}$.
For $\mathcal G\subset\cD$,
write
\[
 h_{\mathcal G}=\sum_{I\in\mathcal G}\1_I,
 \qquad \sh(\mathcal G)=\bigcup_{I\in\mathcal G}I.
\]
The family is \emph{incomparable} if no two distinct members contain
one another, and \emph{$\eta$-sparse} if there are pairwise disjoint
measurable sets $E_I\subset I$ with $|E_I|\geq\eta|I|$.
Rey uses the term \emph{antichain} for an incomparable family.
For fixed $0<\eta<1$, Rey's question is which powers $\alpha>0$ give
\[
 \int_{\R^d}\bigl[e^{c h_{\mathcal G}(x)^\alpha}-1\bigr]\ud x
 \leq C|\sh(\mathcal G)|
\]
uniformly over families with finite shadow.  The constants may depend
on $m$, $d$, and $\eta$, but not on $\mathcal G$.
Sparseness alone gives the sharp power $1/m$ by maximal-function
arguments.  Rey proves that incomparability improves this to power
$1$ when $m=2$, and conjectures the power $1/(m-1)$ for every
$m\geq2$~\cite[Conjecture~2]{Rey}.  We prove this conjecture in
Theorem~\ref{thm:overlap}.  Rey's examples show that these powers
cannot be increased.

\begin{center}
\small
\renewcommand{\arraystretch}{1.25}
\begin{tabular}{@{}lcc@{}}
\hline
\multicolumn{3}{@{}l}{\emph{Sharp exponential powers for unweighted overlap}}\\[3pt]
Setting & Sparse & Sparse and incomparable\\
\hline
Bi-parameter & $1/2$ & $1$ (Rey's theorem)\\
$m$-parameter, $m\geq3$ & $1/m$ & $1/(m-1)$ (this paper)\\[3pt]
\hline
\end{tabular}
\end{center}

For the dyadic Zygmund formulation, let $m\geq3$ and let
$\Phi\colon\Z^{m-1}\to\Z$ be coordinatewise nondecreasing, and
let $\mathcal Z_\Phi\subset\cD$ be the family of \emph{$m$-parameter
dyadic $\Phi$-Zygmund rectangles}, that is, rectangles with side-length tuple
\[
 \bigl(2^{k_1},\ldots,2^{k_{m-1}},
               2^{\Phi(k_1,\ldots,k_{m-1})}\bigr),
 \qquad (k_1,\ldots,k_{m-1})\in\Z^{m-1}.
\]
The dyadic conjecture, explicitly formulated as Conjecture~1 in
Rey~\cite{Rey}, is \eqref{eq:zygmund} with
$M_{\mathcal Z_\Phi}$ in place of $M_{\mathcal R_\phi}$,
uniformly in $\Phi$ and the grids.

The relation between the exponential integrability conjecture for the
overlap $h_{\mathcal G}$ and the dyadic Zygmund conjecture deserves
some explanation. This link was explicitly given by Rey in~\cite{Rey}.
For a finite family $\mathcal F\subset\mathcal Z_\Phi$ on whose
rectangles the average of $|f|$ exceeds $\lambda$, a
C\'ordoba--Fefferman selection~\cite{CordobaFefferman} gives $\mathcal G\subset\mathcal F$
with comparable shadow measure.  Hence
\[
 \lambda|\sh(\mathcal F)|
 \lesssim\lambda|\sh(\mathcal G)|
 \leq\sum_{I\in\mathcal G}\int_I|f|
 =\int|f|h_{\mathcal G}.
\]
The selection has a second important property: its slices
$\mathcal G_u$, obtained by fixing $u\in\R^{d_m}$ in the last
coordinate, are uniformly sparse $(m-1)$-parameter families.
The selected family $\mathcal G$ is also incomparable; monotonicity
of $\Phi$ then makes its slices incomparable as well.  The identity
$h_{\mathcal G}(x,u)=h_{\mathcal G_u}(x)$ allows us to integrate
the conjectured exponential bound for $h_{\mathcal G_u}$ in $u$.
Young's inequality then converts the resulting exponential bound
for $h_{\mathcal G}$ into the dyadic Zygmund estimate.
In the reverse direction, a maximal-function estimate gives an
overlap bound for every sparse family in the same basis, but with a
weaker exponential power.  Writing $\exp L^\alpha$ for the overlap
estimate of power $\alpha$, the two routes give
\[
 \begin{aligned}
 &(m-1)\text{-parameter sparse incomparable }\exp L^{1/(m-2)}
 \\
 &\quad\Longrightarrow
 m\text{-parameter $\Phi$-Zygmund weak }L(\log L)^{m-2}
 \quad\text{(the dyadic Zygmund conjecture)}
 \\
 &\quad\Longrightarrow
 m\text{-parameter $\Phi$-Zygmund sparse }\exp L^{1/(m-1)}.
 \end{aligned}
\]
The first arrow uses the special selections described above; the second applies to all
sparse families in $\mathcal Z_\Phi$.  It does not recover Rey's
estimate for arbitrary sparse incomparable families in the full
product grid, so these implications give no converse to Rey's
reduction.  In this sense, Rey's conjecture is stronger than the
dyadic Zygmund conjecture, and it is this stronger conjecture that
we prove. In Section~\ref{sec:continuous-zygmund} we give the reduction
in full, including the non-trivial extension needed for the non-dyadic
Zygmund conjecture, so the paper is self-contained in its logic.

Notice that the second arrow above says that, for sparse families, the
$\Phi$-Zygmund scale restriction gives the same exponential power
$1/(m-1)$ as incomparability does for general product families.
Section~\ref{sec:phi-sparse} supplies the corresponding sharpness result
for this known implication: imposing incomparability on top of
sparseness and the $\Phi$-Zygmund relation need not improve
this power further, even for the classical relation
$\Phi(k_1,\ldots,k_{m-1})=k_1+\cdots+k_{m-1}$.

Rey proves the $m=2$ case of his conjecture directly. His beautiful
proof orders intersecting incomparable rectangles by increasing
inclusion in one coordinate and decreasing inclusion in the other.
This order factors the measure of a multiple intersection into a
product of normalized pairwise overlaps. Sparseness and the $L^2$
boundedness of the strong maximal function give an $L^2$ estimate
for linear combinations of normalized rectangle indicators.
Together with the intersection formula, this yields moment bounds
for $h_{\mathcal G}$ and hence the exponential estimate.
He then deduces, as an application of this result, the maximal-function
bound \eqref{eq:intro-incomparable-maximal} for $m=2$, with the stated sharp
dependence on $p'=p/(p-1)$.

\subsection*{How our proof works}

We go in the opposite direction. We first prove directly, in
Theorem~\ref{thm:maximal}, that every
incomparable family $\mathcal G\subset\cD$ satisfies
\begin{equation}\label{eq:intro-incomparable-maximal}
 \|M_{\mathcal G}f\|_p
 \leq C_{m,d}(p')^{m-1}\|f\|_p,
 \qquad1<p\leq2.
\end{equation}
No sparseness is assumed here.
The power of $p'$ is sharp as $p$ approaches one, as follows from
Rey's simplex examples~\cite[Section~3.2.2]{Rey}.
The hard part is to prove \eqref{eq:intro-incomparable-maximal}.
Once this is established, a short duality argument gives
$\|h_{\mathcal G}\|_q\leq C_{m,d,\eta}q^{m-1}
|\sh(\mathcal G)|^{1/q}$ for $q\geq2$ whenever $\mathcal G$
is also $\eta$-sparse, and hence Rey's conjectured exponential integrability.
In particular, this also gives a new proof
of the case $m=2$ originally proved by Rey.

To prove \eqref{eq:intro-incomparable-maximal}, we first reduce to
a finite setting with $f\geq0$. For an individual rectangle $R$, removing from $f$
the product martingale differences indexed by subrectangles of $R$
gives a function $F_R$ with $E_RF_R=E_Rf$, where $E_R$ denotes
averaging over $R$. The removed differences ensure that $E_RF_R$
can be expressed using averages of $F_R$ in at most $m-1$ coordinates.
Incomparability allows us to obtain both properties simultaneously
for every $R\in\mathcal G$ with a \textbf{single} function $F$.
We construct it by removing the differences indexed by subrectangles
of any member of $\mathcal G$; see Lemma~\ref{lem:representation}.

Taking the supremum over rectangles now gives a bound by maximal
functions in only $m-1$ coordinates:
\[
 M_{\mathcal G}f(x)\leq C_m\sum_{j=1}^m\mathcal M_{A_j}F(x),
 \qquad A_j=\{1,\ldots,m\}\setminus\{j\}.
\]
Here $\mathcal M_{A_j}$ takes the supremum of absolute values of
averages in the coordinates $A_j$. In contrast to $M$, the absolute
value is taken \emph{after} averaging, so cancellation in the
possibly sign-changing function $F$ is retained. Having a common
$F$ is essential: separate functions $F_R$ would leave a different
input for every rectangle.

Why is this promising? The usual maximal-function bound costs one
factor $p'$ per coordinate. The unrestricted $m$-parameter bound
therefore has $(p')^m$, whereas the representation above suggests
the desired $(p')^{m-1}$. But there is a remaining difficulty.
Applying the ordinary maximal-function bounds gives only
\[
 \|M_{\mathcal G}f\|_p
 \leq C_{m,d}(p')^{m-1}\|F\|_p.
\]
The right power of $p'$ is present, but we need $\|f\|_p$ on the
right. The construction does not provide a bound
$\|F\|_p\leq C_{m,d}\|f\|_p$ uniform as $p$ approaches one.

We instead exploit the cancellation retained by $\mathcal M$.
A dyadic product analogue of the Burkholder--Davis--Gundy
maximal--square-function estimate (Lemma~\ref{lem:square}) controls
the maximal functions of the relevant martingale sums by their
square functions, with a constant uniform as $p$ approaches one.
This reduces the main task to estimating square functions of $F$
in $m-1$ coordinates without an additional loss in $p$.

Why should these square functions be easier to control? The
construction of $F$ also gives
\[
 \Delta_{K^{A_j}}F=X_{K^{A_j}}\Delta_{K^{A_j}}f.
\]
Here $K^{A_j}$ is a product of dyadic cubes in the coordinates
$A_j$, and $\Delta_{K^{A_j}}$ is the corresponding product
martingale difference. The \emph{cross-coordinate averaging operator}
$X_{K^{A_j}}$ averages in the omitted coordinate $j$, over a
partition determined by $K^{A_j}$ and $\mathcal G$;
see Lemma~\ref{lem:cross}. Thus the square functions of $F$ involve
averages of the corresponding martingale differences of $f$.
But another difficulty remains: the averaging operator $X_{K^{A_j}}$
depends on the difference index $K^{A_j}$, so we cannot simply take a common
averaging operator outside the square function.
This final difficulty is overcome by weighted square-function
estimates with precise control of their dependence on $p$.
H\"older's inequality and absorption then yield
\eqref{eq:intro-incomparable-maximal}.

As explained above, \eqref{eq:intro-incomparable-maximal} gives
Rey's conjecture for every $m\geq2$ and, by Rey's reduction, the
dyadic Zygmund conjecture for every $m\geq3$. In passing from
arbitrary side lengths to dyadic covers, rounding can destroy the
incomparability of the slices used in Rey's reduction. To obtain the
non-dyadic version, we therefore prove a more general form of Rey's exponential
integrability estimate: it holds for sparse dyadic families
$\mathcal G$ satisfying the weaker condition
\[
 I,J\in\mathcal G,\quad I\subset J
 \quad\Longrightarrow\quad I^i=J^i
              \text{ for some }i\in\{1,\ldots,m\}.
\]
This extension and its application to the non-dyadic Zygmund
conjecture are proved in Section~\ref{sec:continuous-zygmund}.

\subsection*{Notation and conventions}

We use $\subset$ for inclusion allowing equality and $\subsetneq$
for strict inclusion.  All functions are real-valued; this entails no
loss in the maximal-function estimates, since
$M_{\mathcal E}f=M_{\mathcal E}|f|$.
All measures and integrals are with respect
to Lebesgue measure, and $\|g\|_p$ denotes the $L^p$ norm on the
indicated space.  Empty sums and suprema are zero; we omit the
trivial case of an empty initial rectangle family.
The constants $C_m$, $C_d$, and $C_{m,d}$ depend only on the indicated
parameters and may change between estimates unless fixed explicitly.
We write $A\lesssim B$ when $A\leq C_{m,d}B$.
For $1<p<\infty$, let $p'=p/(p-1)$.  In particular,
\[
 (p-1)^{-1}\leq p'\leq2(p-1)^{-1},\qquad1<p\leq2.
\]
Thus the powers of $p'$ in the maximal-function estimates and of
$(p-1)^{-1}$ in the weighted square-function estimates describe the
same growth as $p$ approaches one.

Let $\cD^i$ be a dyadic grid of cubes in $\R^{d_i}$ and set
$\cD=\cD^1\times\cdots\times\cD^m$ and $d=d_1+\cdots+d_m$.
The cubes in each grid are half-open, and those of side length
$2^k$ partition the corresponding space for every $k\in\Z$.
Any two intersecting cubes in the same grid are nested.
Rectangles are written $I=I^1\times\cdots\times I^m$.
For a cube $I^i\in\cD^i$, let $\ell(I^i)$ denote its side length.
In the special case that $I\in\cD$ is itself a cube, we write
$\ell(I)=\ell(I^1)=\cdots=\ell(I^m)$.
For a dyadic cube $Q$, let $\ch(Q)$ denote its dyadic children,
the subcubes of side length $\ell(Q)/2$.
For a rectangle $I$, define
\[
 \ch(I)=\{J^1\times\cdots\times J^m:
                    J^i\in\ch(I^i),\ 1\leq i\leq m\}.
\]
Thus $|J|=2^{-d}|I|$ for $J\in\ch(I)$.

Write $\langle g\rangle_Q=|Q|^{-1}\int_Qg$ for averages, and put
\[
 E_Qg=\1_Q\langle g\rangle_Q,
 \qquad
 \Delta_Qg=\sum_{J\in\ch(Q)}E_Jg-E_Qg.
\]
On the product space, $E_{Q^i}$ and $\Delta_{Q^i}$ act in the $i$th
coordinate, and $\Id_i$ denotes the identity in that coordinate.
Products of coordinate operators denote composition; in particular,
$E_I=\prod_iE_{I^i}$ and $\Delta_I=\prod_i\Delta_{I^i}$.
For $x=(x_1,\ldots,x_m)\in\R^d$ and
$A\subset\{1,\ldots,m\}$, write
\[
 \begin{gathered}
 x^A=(x_i)_{i\in A},\qquad
 K^A=\prod_{i\in A}K^i,\qquad
 \cD^A=\prod_{i\in A}\cD^i,\\
 E_{K^A}=\prod_{i\in A}E_{K^i},\qquad
 \Delta_{K^A}=\prod_{i\in A}\Delta_{K^i}.
 \end{gathered}
\]
These operators leave the other coordinates unchanged.  An empty
product of operators is the identity $\Id$.

\subsection*{Lean Formalization}

A Lean~4 formalization of the main results is available at
\begin{center}
\url{https://github.com/hjmartik/ZygmundConjecture}
\end{center}
It covers the incomparable maximal and overlap estimates
(Theorems~\ref{thm:maximal} and~\ref{thm:overlap}), their extension
under the weaker containment hypothesis
(Theorem~\ref{thm:weaker-containment}), the continuous and rounded
dyadic Zygmund endpoint estimates
(Theorems~\ref{thm:continuous-zygmund} and~\ref{thm:rounded-endpoint}),
and the sparse $\Phi$-Zygmund overlap estimate and its sharpness under
incomparability
(Corollary~\ref{cor:phi-sparse} and Theorem~\ref{thm:phi-sparse-sharpness}).

The formalized proofs contain no unproved placeholders or custom
axioms. All seven results have independently formulated restatements
whose definitions use only Mathlib; checked Lean proofs connect these
definitions and statements to the main development. Comparator has
confirmed agreement with these independent statements, and the exported
proofs have been checked by both Lean's kernel and the independent checker Nanoda.

The formalization was developed with AI assistance, with separate
reviews of its scope and faithfulness to the intended mathematics.
The main Lean codebase was produced using GPT-6-Astra at Max reasoning
effort, in a workflow with specialized subagents. Claude Fable~5.1
provided additional independent and blind review of the mathematical
specifications.

\subsection*{Acknowledgements}

This material is based upon work supported by the National Science Foundation under Grant
No. 2247234. H.M. was also supported by the Simons Foundation through MP-TSM-00002361
(travel support for mathematicians). The author thanks Kangwei Li for helpful discussions
and comments on the manuscript.
The author used OpenAI's Codex as a research and editorial aid in developing this manuscript---
for example, to explore strategies and formulations, stress-test proofs, and improve the
organization and exposition.

\section{A product maximal square-function estimate}\label{sec:square}

In this section we allow $m\geq1$, including the one-parameter setting.
We distinguish the positive maximal function from the supremum of
absolute averages:
\[
 M_{\cD}g(x)=\sup_{\substack{I\in\cD\\x\in I}}
                 \langle|g|\rangle_I,
 \qquad
 \mathcal M_{\cD}g(x)=\sup_{\substack{I\in\cD\\x\in I}}
                 |\langle g\rangle_I|.
\]
Thus $\mathcal M_{\cD}g\leq M_{\cD}g$, but cancellation in the
averages will be essential below.  The rectangular dyadic square
function is
\[
 S_{\cD}g(x)=\left(\sum_{I\in\cD}|\Delta_Ig(x)|^2\right)^{1/2}.
\]
Iterating the ordinary one-coordinate dyadic maximal inequality gives
\begin{equation}\label{eq:ordinary-maximal}
 \|M_{\cD}g\|_p
 \leq\|M_{\cD^1}\cdots M_{\cD^m}g\|_p
 \leq(p')^m\|g\|_p,
 \qquad 1<p<\infty.
\end{equation}
For a subfamily $\mathcal E\subset\cD$, $M_{\mathcal E}$ denotes the
same positive maximal function with the supremum restricted to
$I\in\mathcal E$.

The following lemma is a dyadic product analogue of the classical
Burkholder--Davis--Gundy maximal--square-function
estimate~\cite[Theorem~1.1]{BurkholderDavisGundy}.
What matters here is that the constant remains bounded as $p$
approaches one.  We give a direct proof for finite sums of product
martingale differences.

\begin{lemma}[Product maximal square-function estimate]\label{lem:square}
Suppose that $g=\sum_{I\in\mathcal H}\Delta_Ig$ for a finite
$\mathcal H\subset\cD$.  Then
\begin{equation}\label{eq:square}
 \|\mathcal M_{\cD}g\|_p\leq C_d\|S_{\cD}g\|_p,
 \qquad 1\leq p\leq\tfrac32,
\end{equation}
where $C_d=2^{6d+10}$.
\end{lemma}

\begin{proof}
All differences outside $\mathcal H$ vanish, so
$(S_{\cD}g)^2=\sum_{I\in\mathcal H}|\Delta_Ig|^2$.
We group these differences according to the size of $S_{\cD}g$ on
their support rectangles.  Set
\[
 \delta=2^{-(d+1)},\qquad
 \Omega_n=\{S_{\cD}g>2^n\},\qquad
 \widetilde\Omega_n=\{M_{\cD}\1_{\Omega_n}>\delta\}.
\]
By \eqref{eq:ordinary-maximal} with $p=2$ and $m\leq d$,
\begin{equation}\label{eq:square-halo}
 |\widetilde\Omega_n|
 \leq\delta^{-2}\|M_{\cD}\1_{\Omega_n}\|_2^2
 \leq4^m\delta^{-2}|\Omega_n|
 \leq2^{4d+5}|\Omega_n|.
\end{equation}
For each integer $n$, define
\[
 \mathcal H_n=\{I\in\mathcal H:
 |I\cap\Omega_n|>\delta|I|,
 \ |I\cap\Omega_{n+1}|\leq\delta|I|\},
 \qquad g_n=\sum_{I\in\mathcal H_n}\Delta_Ig.
\]
For a nonzero $\Delta_Ig$, choose $J\in\ch(I)$ on which
$\Delta_Ig$ is nonzero.
If $2^n<|\langle\Delta_Ig\rangle_J|$, then
\[
 J\subset\{|\Delta_Ig|>2^n\}\subset\Omega_n,
 \qquad |I\cap\Omega_n|\geq|J|=2^{-d}|I|=2\delta|I|.
\]
On the other hand, $\Omega_n=\varnothing$ when
$2^n\geq\|S_{\cD}g\|_\infty$.  Since $\Omega_{n+1}\subset\Omega_n$,
\[
 I\in\mathcal H_n
 \quad\Longleftrightarrow\quad
 n=\max\{k\in\Z:|I\cap\Omega_k|>\delta|I|\}.
\]
Thus every nonzero difference belongs to exactly one sum, and
$g=\sum_n g_n$.

For $I\in\mathcal H_n$, we have
\[
 M_{\cD}\1_{\Omega_n}(x)
 \geq\frac{|I\cap\Omega_n|}{|I|}>\delta\quad(x\in I),
 \qquad I\subset\widetilde\Omega_n.
\]
For every $J\in\ch(I)$,
\[
 |J\cap\Omega_{n+1}|\leq\delta|I|=\tfrac12|J|,
 \qquad |J\setminus\Omega_{n+1}|\geq\tfrac12|J|.
\]
Constancy of $\Delta_Ig$ on $J$ therefore gives
\[
 \int_J|\Delta_Ig|^2
 =\frac{|J|}{|J\setminus\Omega_{n+1}|}
       \int_{J\setminus\Omega_{n+1}}|\Delta_Ig|^2
 \leq2\int_{J\setminus\Omega_{n+1}}|\Delta_Ig|^2.
\]
Sum over $J\in\ch(I)$, and then over $I\in\mathcal H_n$.
Orthogonality and the support of $\Delta_Ig$ in $I$ give
\begin{align}
 \|g_n\|_2^2
 &=\sum_{I\in\mathcal H_n}\|\Delta_Ig\|_2^2
 \leq2\sum_{I\in\mathcal H_n}
       \int_{I\setminus\Omega_{n+1}}|\Delta_Ig|^2\notag\\
 &=2\int_{\widetilde\Omega_n\setminus\Omega_{n+1}}
            \sum_{I\in\mathcal H_n}|\Delta_Ig|^2\notag\\
 &\leq2\int_{\widetilde\Omega_n\setminus\Omega_{n+1}}
               (S_{\cD}g)^2
 \leq2^{2n+3}|\widetilde\Omega_n|\notag\\
 &\leq2^{2n+3}4^m\delta^{-2}|\Omega_n|
 \leq2^{4d+5}2^{2n}|\Omega_n|.\label{eq:square-piece}
\end{align}

We next check where the averages of $g_n$ can be nonzero.  In one
coordinate, dyadicity and cancellation give the identity
\begin{equation}\label{eq:average-difference}
 E_L\Delta_Qg=
 \begin{cases}
 \1_L\Delta_Qg,&L\subsetneq Q,\\
 0,&\text{otherwise}.
 \end{cases}
\end{equation}
Consequently, for $I\in\mathcal H_n$,
\[
 E_R\Delta_Ig(x)\ne0
 \ \Longrightarrow\
 x\in R,\quad R^i\subsetneq I^i\ (1\leq i\leq m)
 \ \Longrightarrow\
 x\in I\subset\widetilde\Omega_n.
\]
Thus
\begin{equation}\label{eq:square-piece-maximal-support}
 \{\mathcal M_{\cD}g_n>0\}\subset\widetilde\Omega_n.
\end{equation}

Fix $N\in\Z$.  To estimate the level set
$\{\mathcal M_{\cD}g>2^N\}$, we separate the terms with $n<N$
from those with $n\geq N$.
If $n\geq N$, then $\widetilde\Omega_n\subset\widetilde\Omega_N$.
For $x\notin\widetilde\Omega_N$ and $n\geq N$,
\eqref{eq:square-piece-maximal-support} therefore gives
$\mathcal M_{\cD}g_n(x)=0$.  Consequently, for every $R\in\cD$,
\[
 |E_Rg_n(x)|
 =\1_R(x)|\langle g_n\rangle_R|
 \leq\mathcal M_{\cD}g_n(x)=0.
\]
Thus
\[
 \mathcal M_{\cD}g(x)
 =\mathcal M_{\cD}\!\left(\sum_{n<N}g_n\right)(x)
 \leq M_{\cD}\!\left(\sum_{n<N}g_n\right)(x),
 \qquad x\notin\widetilde\Omega_N.
\]
By orthogonality,
$\|\sum_{n<N}g_n\|_2^2=\sum_{n<N}\|g_n\|_2^2$.
Chebyshev's inequality and \eqref{eq:ordinary-maximal} now yield
\begin{align}
 |\{\mathcal M_{\cD}g>2^N\}|
 &\leq|\widetilde\Omega_N|
       +2^{-2N}\left\|M_{\cD}\!\left(\sum_{n<N}g_n\right)\right\|_2^2
       \notag\\
 &\leq|\widetilde\Omega_N|
       +4^m2^{-2N}\sum_{n<N}\|g_n\|_2^2\notag\\
 &\leq2^{6d+5}\left(|\Omega_N|
       +2^{-2N}\sum_{n<N}2^{2n}|\Omega_n|\right).
       \label{eq:square-distribution}
\end{align}
Multiply by $2^{pN}$ and sum over $N$.  For the second term,
interchange the sums and set $j=N-n$:
\begin{align*}
 \sum_N2^{(p-2)N}\sum_{n<N}2^{2n}|\Omega_n|
 &=\sum_n2^{2n}|\Omega_n|\sum_{N>n}2^{(p-2)N}\\
 &=\sum_n2^{pn}|\Omega_n|
       \sum_{j\geq1}2^{-(2-p)j}\\
 &=\frac1{2^{2-p}-1}\sum_n2^{pn}|\Omega_n|
 \leq3\sum_n2^{pn}|\Omega_n|,
\end{align*}
where $1\leq p\leq3/2$.  Finally, for a nonnegative $H$,
\[
 \tfrac14 H(x)^p
 \leq\sum_{2^n<H(x)}2^{pn}
 \leq2H(x)^p.
\]
Integration and \eqref{eq:square-distribution} therefore give
\begin{align*}
 \|\mathcal M_{\cD}g\|_p^p
 &\leq4\sum_N2^{pN}|\{\mathcal M_{\cD}g>2^N\}|\\
 &\leq2^{6d+9}\sum_n2^{pn}|\Omega_n|\\
 &\leq2^{6d+10}\|S_{\cD}g\|_p^p.
\end{align*}
Taking $p$th roots and using $p\geq1$ proves \eqref{eq:square}
with the stated constant.
\end{proof}

\section{Maximal estimates and sparse overlap}\label{sec:maximal}

We reduce the incomparable maximal estimate to a finite, localized
setting and then deduce the sparse-overlap theorem by duality.
The main technical work is carried out in
Sections~\ref{sec:geometry} and~\ref{sec:square-estimates}, which
establish the finite estimate in Lemma~\ref{lem:represented-bound}.

\subsection{The incomparable maximal estimate}

\begin{theorem}[Incomparable maximal estimate]\label{thm:maximal}
Let $m\geq2$, and let $\mathcal G$ be an incomparable family in
$\cD^1\times\cdots\times\cD^m$ on
$\R^{d_1}\times\cdots\times\R^{d_m}$.  Then
\begin{equation}\label{eq:maximal-theorem}
 \|M_{\mathcal G}f\|_p
 \leq C_{m,d}(p')^{m-1}\|f\|_p,
 \qquad1<p\leq2.
\end{equation}
The constant depends only on $m$ and $d$.
\end{theorem}

\begin{proof}
It suffices to consider finite $\mathcal G$: if
$\mathcal G_n\uparrow\mathcal G$ through finite subfamilies, then
$M_{\mathcal G_n}f\uparrow M_{\mathcal G}f$, so monotone convergence
passes the uniform estimate to $\mathcal G$.

Choose $a,k\in\Z$ so that
\[
 2^{-k}\leq\ell(I^i)\leq2^a,
 \qquad I\in\mathcal G,\quad1\leq i\leq m.
\]
The cubes $I_0\in\cD$ with $\ell(I_0)=2^a$
are pairwise disjoint, and each $I\in\mathcal G$ is contained in
exactly one of them.
Set $\mathcal G(I_0)=\{I\in\mathcal G:I\subset I_0\}$.  Since
the maximal function associated with this subfamily is supported
on $I_0$,
\begin{equation}\label{eq:large-cube-localization}
 M_{\mathcal G}f
 =\sum_{I_0}M_{\mathcal G(I_0)}(f\1_{I_0}),
 \qquad
 \|M_{\mathcal G}f\|_p^p
 =\sum_{I_0}\|M_{\mathcal G(I_0)}(f\1_{I_0})\|_p^p.
\end{equation}

Fix one $I_0$ and define, summing over cubes $J$,
\[
 \widetilde f
 =\sum_{\substack{J\in\cD,\ J\subset I_0\\
                  \ell(J)=2^{-k}}}E_J|f|.
\]
Every $I\in\mathcal G(I_0)$ is a union of these cubes.
Summing over those $J\subset I$ gives
\[
 \langle\widetilde f\rangle_I
 =\frac1{|I|}\sum_J\int_J\widetilde f
 =\frac1{|I|}\sum_J|J|\langle|f|\rangle_J
 =\langle|f|\rangle_I.
\]
Consequently,
\begin{equation}\label{eq:input-averaging}
 M_{\mathcal G(I_0)}f=M_{\mathcal G(I_0)}\widetilde f,
 \qquad \|\widetilde f\|_p\leq\|f\1_{I_0}\|_p.
\end{equation}
For $\varepsilon>0$, the function
$\widetilde f+\varepsilon\1_{I_0}$ is strictly positive on $I_0$
and constant on the smallest product cubes.  Thus
\eqref{eq:input-averaging} and Lemma~\ref{lem:represented-bound} give,
for $1<p\leq3/2$,
\begin{align*}
 \|M_{\mathcal G(I_0)}f\|_p
 &=\|M_{\mathcal G(I_0)}\widetilde f\|_p\\
 &\leq\|M_{\mathcal G(I_0)}
             (\widetilde f+\varepsilon\1_{I_0})\|_p\\
 &\leq C_{m,d}(p')^{m-1}\|\widetilde f+\varepsilon\1_{I_0}\|_p.
\end{align*}
Letting $\varepsilon\downarrow0$ gives
\[
 \|M_{\mathcal G(I_0)}f\|_p
 \leq C_{m,d}(p')^{m-1}\|\widetilde f\|_p
 \leq C_{m,d}(p')^{m-1}\|f\1_{I_0}\|_p.
\]
For $3/2\leq p\leq2$,
the ordinary product maximal bound gives $(p')^m\leq3^m$, and hence
the same estimate in the remaining range.  Summing by
\eqref{eq:large-cube-localization} gives
\[
 \|M_{\mathcal G}f\|_p^p
 \leq C_{m,d}^p(p')^{(m-1)p}\sum_{I_0}\|f\1_{I_0}\|_p^p
 =C_{m,d}^p(p')^{(m-1)p}\|f\|_p^p.
\]
\end{proof}

\subsection{Sparse incomparable overlap}\label{sec:consequences}

The maximal estimate in Theorem~\ref{thm:maximal} gives the following
overlap bounds by duality.
Rey~\cite{Rey} states his conjecture for $1/2$-sparse families;
the estimate below allows any sparsity parameter $\eta\in(0,1]$.

\begin{theorem}[Sparse incomparable overlap]\label{thm:overlap}
Let $\mathcal G\subset\cD$ be incomparable and $\eta$-sparse,
with finite shadow.  Then
\begin{equation}\label{eq:overlap-moments}
 \|h_{\mathcal G}\|_q
 \leq C_{m,d}\eta^{-1}q^{m-1}|\sh(\mathcal G)|^{1/q},
 \qquad q\geq2,
\end{equation}
and
\begin{equation}\label{eq:overlap-exponential}
 \int_{\sh(\mathcal G)}
 \left[\exp\!\left(c_{m,d}(\eta h_{\mathcal G})^{1/(m-1)}\right)-1\right]
 \leq C_{m,d}|\sh(\mathcal G)|.
\end{equation}
The constants depend only on $m$ and $d$.
\end{theorem}

\begin{proof}
First take $\mathcal G$ finite.  For $g\geq0$, the disjoint sets
$E_I$ from sparseness and Theorem~\ref{thm:maximal} with $p=q'$
give
\begin{align*}
 \int h_{\mathcal G}g
 &=\sum_{I\in\mathcal G}|I|\langle g\rangle_I
 \leq\eta^{-1}\sum_{I\in\mathcal G}\int_{E_I}M_{\mathcal G}g\\
 &\leq\eta^{-1}\int_{\sh(\mathcal G)}M_{\mathcal G}g
 \leq C_{m,d}\eta^{-1}q^{m-1}
           |\sh(\mathcal G)|^{1/q}\|g\|_{q'}.
\end{align*}
Duality proves \eqref{eq:overlap-moments}.

For the exponential estimate, write $H=\eta h_{\mathcal G}$.
Chebyshev's inequality gives
\[
 \frac{|\{x\in\sh(\mathcal G):H(x)>t\}|}{|\sh(\mathcal G)|}
 \leq\frac{\|H\|_q^q}{t^q|\sh(\mathcal G)|}
 \leq\left(\frac{C_{m,d}q^{m-1}}t\right)^q.
\]
For sufficiently large $t$, choose
$q=(t/(e^2C_{m,d}))^{1/(m-1)}\geq2$.  Then
$(C_{m,d}q^{m-1}/t)^q=(e^{-2})^q=e^{-2q}$.
For smaller $t$, use the trivial bound
$|\{H>t\}|\leq|\sh(\mathcal G)|$.
Thus there are constants $c_0,C_0>0$, depending only on $m$ and $d$,
such that
\[
 |\{H>t\}|
 \leq C_0\exp(-c_0t^{1/(m-1)})|\sh(\mathcal G)|,
 \qquad t\geq0.
\]
For $0<c<c_0$, write the exponential as an integral and interchange
the integrals:
\begin{align*}
 \int_{\sh(\mathcal G)}(e^{cH^{1/(m-1)}}-1)
 &=\int_{\sh(\mathcal G)}\int_0^{H(x)^{1/(m-1)}}ce^{cs}\ud s\ud x\\
 &=\int_0^\infty ce^{cs}
       |\{x:H(x)^{1/(m-1)}>s\}|\ud s\\
 &\leq C_0c\int_0^\infty e^{-(c_0-c)s}\ud s
       |\sh(\mathcal G)|.
\end{align*}
This proves \eqref{eq:overlap-exponential}.  Finite exhaustion and
monotone convergence extend both estimates to the stated families.
\end{proof}

\section{Cross-coordinate averaging and the projection}
\label{sec:geometry}

For this section and the next, fix $\mathcal G,f,I_0,k$ under the
assumptions of Lemma~\ref{lem:represented-bound}.  The reduction in
Section~\ref{sec:maximal} has already brought us to this setting.
Here we develop the representation used to prove the estimate in
Section~\ref{sec:square-estimates}.
The averages $E_If$, $I\in\mathcal G$, are the quantities we want
to estimate.  We construct a function $F$ with $E_IF=E_If$ for
$I\in\mathcal G$, and with $\Delta_LF=0$ whenever $L\in\cD$
satisfies $L\subset I$ for some $I\in\mathcal G$.
This gives a representation of $E_If$ using
averages in at most $m-1$ coordinates.  Cross-coordinate averaging
will describe the remaining differences of $F$.
The final subsection uses Lemma~\ref{lem:square} to reduce the
maximal-function estimate to square functions of $F$.

\subsection{Finite setting and the projection}

Set
\[
 \cD_0^i=\{Q^i\in\cD^i:Q^i\subset I_0^i,
                         \ \ell(Q^i)\geq2^{-k}\},
 \qquad
 (\cD_0^i)^\circ=\{Q^i\in\cD_0^i:\ell(Q^i)>2^{-k}\}.
\]
The original grids $\cD^i$ remain unchanged.  We use the finite
collections only in this construction.  All norms and integrals in
these two sections are on $I_0$ or the coordinate factors indicated.
Functions are extended by zero outside $I_0$.

For $A\subset\{1,\ldots,m\}$, write
\[
 \cD_0^A=\prod_{i\in A}\cD_0^i,\qquad
 (\cD_0^A)^\circ=\prod_{i\in A}(\cD_0^i)^\circ.
\]
When $A=\{1,\ldots,m\}$, we omit the superscript $A$; in particular,
the full product families are denoted by $\cD_0$ and $\cD_0^\circ$.

\begin{definition}[The projection]\label{def:projection}
Define
\begin{equation}\label{eq:projection}
\begin{aligned}
 F&:=f-\sum_{\substack{L\in\cD\\
                  L\subset I\text{ for some }I\in\mathcal G}}
                  \Delta_Lf\\
  &=f-\sum_{\substack{L\in\cD_0^\circ\\
                  L\subset I\text{ for some }I\in\mathcal G}}
                  \Delta_Lf.
\end{aligned}
\end{equation}
Each rectangle $L$ occurs once, even if it is contained in several
members of $\mathcal G$.
\end{definition}

For $L\in\cD$ with $L\subset I_0$, constancy of $f$ on the
smallest cubes gives
\[
 \ell(L^i)\leq2^{-k}\text{ for some }i
 \quad\Longrightarrow\quad \Delta_Lf=0.
\]
This proves the second equality in \eqref{eq:projection}.
Each $\Delta_Lf$ in the finite sum is constant on the rectangles
in $\ch(L)$, which are unions of smallest cubes since
$L\in\cD_0^\circ$.  Hence $F$ is also constant on the smallest
cubes.
Thus averages are indexed by $\cD_0$, while restricting difference
sums to $\cD_0^\circ$ omits only zero terms.

The function $F$ need not be nonnegative.  Orthogonality and
cancellation, together with the preceding observation, give
\begin{equation}\label{eq:projection-properties}
\begin{aligned}
 E_{I_0^i}F&=E_{I_0^i}f &&(1\leq i\leq m),\\
 \Delta_LF&=0 &&(L\in\cD,\ L\subset I\in\mathcal G).
\end{aligned}
\end{equation}

We now verify the averaging property promised above:
\begin{equation}\label{eq:original-averages}
 E_IF=E_If,\qquad I\in\mathcal G.
\end{equation}
Fix $I\in\mathcal G$ and an index $L\subset J\in\mathcal G$
from the sum in \eqref{eq:projection}.  By \eqref{eq:average-difference},
\[
\begin{aligned}
 E_I\Delta_Lf\ne0
 &\quad\Longrightarrow\quad
 I^i\subsetneq L^i\subset J^i\quad(1\leq i\leq m)\\
 &\quad\Longrightarrow\quad I\subsetneq J.
\end{aligned}
\]
Incomparability excludes this.  Thus every subtracted difference
has zero average on $I$, proving \eqref{eq:original-averages}.

\subsection{The cross-coordinate averages}

We next express the martingale differences of $F$ in $m-1$ coordinates
as averages, in the remaining coordinate, of the corresponding
differences of $f$.

For $j\in\{1,\ldots,m\}$, put
$A_j:=\{1,\ldots,m\}\setminus\{j\}$.
For $K^{A_j}\in\cD_0^{A_j}$, set
\[
 \mathcal A_{K^{A_j}}^j
 =\{I^j:I\in\mathcal G,\ K^{A_j}\subset I^{A_j}\}
 \ \cup\ \{Q^j\in\cD_0^j:\ell(Q^j)=2^{-k}\}.
\]
Let $\mathcal P_{K^{A_j}}^j$ consist of the maximal cubes of
$\mathcal A_{K^{A_j}}^j$.  Since $\mathcal A_{K^{A_j}}^j$ contains all the smallest cubes,
$\mathcal P_{K^{A_j}}^j$ partitions $I_0^j$.  For a function of $x_j$,
define
\begin{equation}\label{eq:cross-average}
 X_{K^{A_j}}g
 =\sum_{P^j\in\mathcal P_{K^{A_j}}^j}E_{P^j}g.
\end{equation}
On the product space it acts in the $j$th coordinate.  This is a
cross-coordinate averaging operator: its partition is determined by
the cubes $K^i$ in the other coordinates.  Both the partition and
the operator depend on the fixed family $\mathcal G$.

\begin{lemma}[Differences of the projected function]\label{lem:cross}
For every $j$ and $K^{A_j}\in(\cD_0^{A_j})^\circ$ as above,
\begin{equation}\label{eq:cross-projection}
 \Delta_{K^{A_j}}F=X_{K^{A_j}}\Delta_{K^{A_j}}f.
\end{equation}
\end{lemma}

\begin{proof}
Applying $\Delta_{K^{A_j}}$ to \eqref{eq:projection} gives
\[
 \Delta_{K^{A_j}}F=\Delta_{K^{A_j}}f
 -\sum_{\substack{L\in\cD_0^\circ\\
                  L\subset I\text{ for some }I\in\mathcal G}}
       \Delta_{K^{A_j}}\Delta_Lf.
\]
For each $i\in A_j$, orthogonality gives
\[
 \Delta_{K^i}\Delta_{L^i}=0\quad(K^i\ne L^i),
 \qquad \Delta_{K^i}^2=\Delta_{K^i}.
\]
Thus only the indices with $L^{A_j}=K^{A_j}$ remain.  For these
indices, the original condition $L\subset I$ is precisely
$L^j\subset I^j$ and $K^{A_j}\subset I^{A_j}$.  Hence
\begin{equation}\label{eq:projected-difference-sum}
 \Delta_{K^{A_j}}F=\Delta_{K^{A_j}}f
 -\sum_{\substack{L^j\in(\cD_0^j)^\circ\\
       L^j\subset I^j\text{ for some }I\in\mathcal G\\
       \text{with }K^{A_j}\subset I^{A_j}}}
       \Delta_{L^j}\Delta_{K^{A_j}}f.
\end{equation}
We now identify the indices in this sum using the partition
$\mathcal P_{K^{A_j}}^j$:
\begin{equation}\label{eq:partition-indices}
\begin{aligned}
 &\{L^j\in(\cD_0^j)^\circ:
       L^j\subset I^j\text{ for some }I\in\mathcal G
       \text{ with }K^{A_j}\subset I^{A_j}\}\\
 &=\{L^j\in(\cD_0^j)^\circ:
         L^j\subset P^j\text{ for some }P^j\in\mathcal P_{K^{A_j}}^j\}.
\end{aligned}
\end{equation}
For the inclusion from left to right, $I^j\in\mathcal A_{K^{A_j}}^j$,
so $I^j\subset P^j$ for some $P^j\in\mathcal P_{K^{A_j}}^j$.
Conversely, suppose $L^j\in(\cD_0^j)^\circ$ and
$L^j\subset P^j\in\mathcal P_{K^{A_j}}^j$.  Then
\[
 \ell(P^j)\geq\ell(L^j)>2^{-k}.
\]
Thus $P^j$, which is a member of $\mathcal A_{K^{A_j}}^j$,
cannot be one of the inserted smallest cubes.  By the definition
of that collection, $P^j=I^j$ for some $I\in\mathcal G$ with
$K^{A_j}\subset I^{A_j}$.  Hence $L^j\subset P^j=I^j$,
proving the inclusion from right to left.

For a function $g$ constant on the smallest cubes in coordinate $j$,
the martingale telescope on any $P^j\in\cD_0^j$ gives
\begin{equation}\label{eq:one-coordinate-telescope}
 \sum_{\substack{L^j\subset P^j\\ \ell(L^j)>2^{-k}}}\Delta_{L^j}g
 =\sum_{\substack{Q^j\subset P^j\\ \ell(Q^j)=2^{-k}}}E_{Q^j}g-E_{P^j}g
 =\1_{P^j}g-E_{P^j}g.
\end{equation}
The second equality uses the constancy of $g$ on each smallest
cube $Q^j$.  Summing over $P^j\in\mathcal P_{K^{A_j}}^j$ gives,
on $I_0^j$,
\begin{align*}
 g-X_{K^{A_j}}g
 &=\sum_{P^j\in\mathcal P_{K^{A_j}}^j}
        \1_{P^j}(g-\langle g\rangle_{P^j})\\
 &=\sum_{P^j\in\mathcal P_{K^{A_j}}^j}
       \sum_{\substack{L^j\in(\cD_0^j)^\circ\\L^j\subset P^j}}
               \Delta_{L^j}g\\
 &=\sum_{\substack{L^j\in(\cD_0^j)^\circ\\
       L^j\subset I^j\text{ for some }I\in\mathcal G\\
       \text{with }K^{A_j}\subset I^{A_j}}}\Delta_{L^j}g.
\end{align*}
Use $g=\Delta_{K^{A_j}}f$ in
\eqref{eq:projected-difference-sum}:
\[
 \Delta_{K^{A_j}}F
 =\Delta_{K^{A_j}}f-\bigl(\Delta_{K^{A_j}}f-X_{K^{A_j}}\Delta_{K^{A_j}}f\bigr)
 =X_{K^{A_j}}\Delta_{K^{A_j}}f.
\]
This proves \eqref{eq:cross-projection}.
\end{proof}

\subsection{Which averages are preserved}

We now extend \eqref{eq:original-averages} to the additional rectangles
needed in the square-function estimate, and obtain a representation valid
for all of them.

\begin{definition}[Averaging rectangles]\label{def:averaging-rectangles}
Define the family of averaging rectangles by
\begin{equation}\label{eq:averaging-rectangles}
 \mathcal B:=\{R\in\cD_0:
 R^j\in\mathcal P_{R^{A_j}}^j
 \text{ for some }j\in\{1,\ldots,m\}\}.
\end{equation}
Thus, in at least one coordinate, $R^j$ is a partition cube for
the cross-coordinate average determined by the other coordinates of $R$.
\end{definition}

\begin{lemma}\label{lem:original-averaging-rectangles}
For every $I\in\mathcal G$,
\[
 I^j\in\mathcal P_{I^{A_j}}^j,\qquad 1\leq j\leq m.
\]
In particular, $\mathcal G\subset\mathcal B$.
\end{lemma}

\begin{proof}
Fix $I\in\mathcal G$ and $j\in\{1,\ldots,m\}$.
We have $I^j\in\mathcal A_{I^{A_j}}^j$.
A strictly larger cube in $\mathcal A_{I^{A_j}}^j$
would be $J^j$ for some $J\in\mathcal G$ with $I^i\subset J^i$
for $i\in A_j$, giving $I\subsetneq J$.  Incomparability excludes
this, so $I^j\in\mathcal P_{I^{A_j}}^j$ and $I\in\mathcal B$.
\end{proof}

We will estimate $M_{\mathcal B}f$, which dominates $M_{\mathcal G}f$.
The enlargement allows this maximal function to control the
cross-coordinate averages, while retaining the representation of
averages proved below.  These two properties will be used together
in the final estimate.

The family $\mathcal B$ need not be incomparable.  Its rectangles
cover $I_0$.  Indeed, fix $j$.  Since $\mathcal P_{I_0^{A_j}}^j$
partitions $I_0^j$, we have
\[
 I_0
 =\bigcup_{\substack{R\in\cD_0,\ R^{A_j}=I_0^{A_j}\\
                      R^j\in\mathcal P_{I_0^{A_j}}^j}}R
 \subset\bigcup_{R\in\mathcal B}R
 \subset I_0.
\]

For $K^{A_j}\in\cD_0^{A_j}$ and $x\in K^{A_j}\times I_0^j$,
let $P^j\in\mathcal P_{K^{A_j}}^j$ be the partition cube containing
$x_j$, and put $R=K^{A_j}\times P^j$.  Then $x\in R\in\mathcal B$,
and \eqref{eq:cross-average} gives
\begin{equation}\label{eq:averaged-denominator}
 X_{K^{A_j}}E_{K^{A_j}}f(x)
 =E_{P^j}E_{K^{A_j}}f(x)=E_Rf(x)\leq M_{\mathcal B}f(x).
\end{equation}

\begin{lemma}[Preserved averages and representation]\label{lem:representation}
For every $R\in\mathcal B$ and $x\in R$,
\begin{equation}\label{eq:preservation}
 E_RF(x)=E_Rf(x).
\end{equation}
Moreover,
\begin{equation}\label{eq:representation}
 E_Rf(x)=\sum_{B\subsetneq\{1,\ldots,m\}}
       (-1)^{m+1+|B|}\left(\prod_{i\in B}E_{R^i}\right)F(x).
\end{equation}
\end{lemma}

\begin{proof}
To prove $E_RF=E_Rf$, we show that each difference removed in
\eqref{eq:projection} has zero average on $R$.
Fix a removed index $L\subset I\in\mathcal G$.  By
\eqref{eq:average-difference}, $E_R\Delta_Lf$ can be nonzero only if
\begin{equation}\label{eq:nonzero-average-condition}
 R^i\subsetneq L^i\quad(1\leq i\leq m).
\end{equation}
Choose $j$ with $R^j\in\mathcal P_{R^{A_j}}^j$.
By \eqref{eq:nonzero-average-condition} and $L\subset I$,
\[
 R^{A_j}\subset I^{A_j},\qquad R^j\subsetneq I^j.
\]
The first inclusion gives $I^j\in\mathcal A_{R^{A_j}}^j$;
the second contradicts the maximality of
$R^j\in\mathcal P_{R^{A_j}}^j$.
This contradiction proves \eqref{eq:preservation}.

To prove \eqref{eq:representation}, apply
\eqref{eq:one-coordinate-telescope} in each coordinate.  Since $F$
is constant on the smallest cubes in each coordinate, this gives
\begin{equation}\label{eq:local-product-expansion}
\begin{aligned}
 \prod_{i=1}^m(\Id_i-E_{R^i})F(x)
 &=\prod_{i=1}^m\Bigg(
       \sum_{\substack{L^i\in(\cD_0^i)^\circ\\L^i\subset R^i}}
             \Delta_{L^i}\Bigg)F(x)\\
 &=\sum_{\substack{L\in\cD_0^\circ\\L\subset R}}
          \Delta_LF(x),\qquad x\in R.
\end{aligned}
\end{equation}
If $\ell(R^j)>2^{-k}$, then $R^j=I^j$ for some
$I\in\mathcal G$ with $R^{A_j}\subset I^{A_j}$.
Every index in the sum therefore satisfies
\[
 L\subset R\subset I\in\mathcal G
 \quad\Longrightarrow\quad \Delta_LF=0
\]
by \eqref{eq:projection-properties}.  If $\ell(R^j)=2^{-k}$,
there is no index $L^j\in(\cD_0^j)^\circ$ with
$L^j\subset R^j$, so the sum is empty.
In both cases, the sum on the right of
\eqref{eq:local-product-expansion} is zero for $x\in R$.
To expand the left side, choose $-E_{R^i}$ from the factors with
$i\in B$ and $\Id_i$ from all other factors.  This contributes
$|B|$ minus signs.  Thus, for $x\in R$,
\[
\begin{aligned}
 0&=\prod_{i=1}^m(\Id_i-E_{R^i})F(x)\\
  &=\sum_{B\subset\{1,\ldots,m\}}(-1)^{|B|}
       \left(\prod_{i\in B}E_{R^i}\right)F(x)\\
  &=(-1)^mE_RF(x)
   +\sum_{B\subsetneq\{1,\ldots,m\}}(-1)^{|B|}
       \left(\prod_{i\in B}E_{R^i}\right)F(x).
\end{aligned}
\]
The empty set contributes $F(x)$; in the last line we have separated
the full set, which contributes $(-1)^mE_RF(x)$.
Moving the other terms to the opposite side and dividing by $(-1)^m$
multiplies their coefficients by $(-1)^{m+1}$.  Using
\eqref{eq:preservation}, we obtain
\[
 E_Rf(x)=E_RF(x)
 =\sum_{B\subsetneq\{1,\ldots,m\}}(-1)^{m+1+|B|}
       \left(\prod_{i\in B}E_{R^i}\right)F(x),\qquad x\in R,
\]
which is \eqref{eq:representation}.
\end{proof}

\subsection{Reduction to square functions}

We now combine the representation \eqref{eq:representation} with
Lemma~\ref{lem:square} to estimate the maximal function by square
functions of $F$.

Since $f>0$ and $\mathcal B$ covers $I_0$,
we have $M_{\mathcal B}f>0$ on $I_0$.  For $j=1,\ldots,m$, set
\begin{equation}\label{eq:projected-square-functions}
 S_{A_j}F=\left(\sum_{K^{A_j}\in(\cD_0^{A_j})^\circ}|\Delta_{K^{A_j}}F|^2\right)^{1/2},
 \qquad
 WF=\left(\sum_{j=1}^m(S_{A_j}F)^2\right)^{1/2}.
\end{equation}

\begin{lemma}[Reduction to square functions]\label{lem:maximal-to-square}
For $1<p\leq3/2$,
\begin{equation}\label{eq:maximal-to-square}
 \|M_{\mathcal B}f\|_p\leq C_{m,d}(p')^{m-2}\|f\|_p+C_{m,d}\|WF\|_p.
\end{equation}
The constant depends only on $m$ and $d$.
\end{lemma}

\begin{proof}
We estimate $M_{\mathcal B}f$ using \eqref{eq:representation}.
For a nonempty $A\subset\{1,\ldots,m\}$ and a function $g$ on $I_0$,
define the signed maximal function in those coordinates by
\[
 \mathcal M_Ag(x)=
 \sup_{\substack{L^A\in\cD_0^A\\x^A\in L^A}}
       |E_{L^A}g(x)|,\qquad x\in I_0.
\]
The supremum includes the smallest cubes, so for $F$ it also bounds
averages in fewer coordinates.  Indeed, let $B\subset A$, $L^B\in\cD_0^B$,
and $x^B\in L^B$.  For $i\in A\setminus B$, choose
the smallest cube $Q^i$ containing $x_i$.  Constancy of $F$ on the
smallest product cubes gives
\[
 |E_{L^B}F(x)|
 =\left|E_{L^B}\prod_{i\in A\setminus B}E_{Q^i}F(x)\right|
 \leq\mathcal M_AF(x).
\]

For each proper
coordinate subset $B$, choose $j\notin B$.  Then $B\subset A_j$,
so the preceding observation gives
\[
 \left|\left(\prod_{i\in B}E_{R^i}\right)F(x)\right|
 \leq\mathcal M_{A_j}F(x),\qquad x\in R.
\]
Since $f>0$ on $I_0$, \eqref{eq:preservation} gives
$E_R|f|(x)=E_Rf(x)=E_RF(x)\geq0$ for $R\in\mathcal B$ and $x\in R$.
In particular, using \eqref{eq:representation},
\begin{equation}\label{eq:representation-maximal}
\begin{aligned}
 M_{\mathcal B}f(x)&=\sup_{\substack{R\in\mathcal B\\x\in R}}|E_RF(x)|\\
 &\leq\sup_{\substack{R\in\mathcal B\\x\in R}}
       \sum_{B\subsetneq\{1,\ldots,m\}}
       \left|\left(\prod_{i\in B}E_{R^i}\right)F(x)\right|\\
 &\leq2^m\sum_{j=1}^m\mathcal M_{A_j}F(x),
       \qquad x\in I_0.
\end{aligned}
\end{equation}
Fix $j$.  To estimate $\mathcal M_{A_j}F$, define
\[
 g_j:=\sum_{K^{A_j}\in(\cD_0^{A_j})^\circ}\Delta_{K^{A_j}}F
\]
and write $F=g_j+(F-g_j)$.
Applying the telescoping identity \eqref{eq:one-coordinate-telescope} in each coordinate
$i\in A_j$ and expanding the product gives, for $x\in I_0$,
\[
\begin{aligned}
 g_j(x)&=\prod_{i\in A_j}(\Id_i-E_{I_0^i})F(x)\\
 &=F(x)+\sum_{\varnothing\ne B\subset A_j}(-1)^{|B|}
            \left(\prod_{i\in B}E_{I_0^i}\right)F(x).
\end{aligned}
\]
For each nonempty $B$, choose $i_0\in B$.
The identities in \eqref{eq:projection-properties} imply
\[
 \left(\prod_{i\in B}E_{I_0^i}\right)F
 =\left(\prod_{i\in B\setminus\{i_0\}}E_{I_0^i}\right)E_{I_0^{i_0}}f
 =\left(\prod_{i\in B}E_{I_0^i}\right)f.
\]
Substituting these identities gives
\begin{equation}\label{eq:top-average-decomposition}
 g_j(x)=F(x)
   +\sum_{\varnothing\ne B\subset A_j}(-1)^{|B|}
           \left(\prod_{i\in B}E_{I_0^i}\right)f(x),
 \qquad x\in I_0.
\end{equation}

Subadditivity of $\mathcal M_{A_j}$ and the triangle inequality give
\[
 \|\mathcal M_{A_j}F\|_p
 \leq\|\mathcal M_{A_j}g_j\|_p
   +\sum_{\varnothing\ne B\subset A_j}
       \left\|\mathcal M_{A_j}
          \left(\prod_{i\in B}E_{I_0^i}f\right)\right\|_p.
\]
We first estimate $\mathcal M_{A_j}g_j$.  By orthogonality,
\[
 S_{\cD^{A_j}}g_j
 =\left(\sum_{K^{A_j}\in(\cD_0^{A_j})^\circ}|\Delta_{K^{A_j}}F|^2\right)^{1/2}
 =S_{A_j}F.
\]
Applying the $(m-1)$-parameter case of Lemma~\ref{lem:square}
on $\cD^{A_j}$, fiberwise with $x_j$ fixed, gives
\[
 \|\mathcal M_{A_j}g_j\|_p
 \leq C_d\|S_{A_j}F\|_p.
\]

For the remaining sum, fix a nonempty $B\subset A_j$.  The function
$E_{I_0^B}f=\prod_{i\in B}E_{I_0^i}f$ has already been averaged
over the top rectangle $I_0^B$.  Every rectangle $L^B$ allowed in
the supremum defining $\mathcal M_{A_j}$ lies inside $I_0^B$.
Thus, for $x\in I_0$,
\[
 E_{L^B}E_{I_0^B}f(x)=E_{I_0^B}f(x),
 \qquad x^B\in L^B\subset I_0^B.
\]
The choices of cubes in these coordinates therefore do not affect the
supremum, and
\[
 \mathcal M_{A_j}E_{I_0^B}f(x)
 \leq M_{\cD^{A_j\setminus B}}E_{I_0^B}f(x),\qquad x\in I_0.
\]
When $B=A_j$, omit the maximal operator
$M_{\cD^{A_j\setminus B}}$ on the right.
Apply \eqref{eq:ordinary-maximal} in the remaining
$|A_j\setminus B|\leq m-2$ coordinates, and use that
$E_{I_0^B}$ is an $L^p$ contraction:
\[
 \|\mathcal M_{A_j}E_{I_0^B}f\|_p
 \leq(p')^{|A_j\setminus B|}\|E_{I_0^B}f\|_p
 \leq(p')^{m-2}\|f\|_p.
\]
Combining these bounds gives
\[
 \|\mathcal M_{A_j}F\|_p
 \leq C_d\|S_{A_j}F\|_p
       +(2^{m-1}-1)(p')^{m-2}\|f\|_p.
\]
Finally, \eqref{eq:representation-maximal} reduces the desired bound
to the sum of these estimates over $j$.  Since $S_{A_j}F\leq WF$
for every $j$, we obtain
\[
\begin{aligned}
 \|M_{\mathcal B}f\|_p
 &\leq2^m\sum_{j=1}^m\|\mathcal M_{A_j}F\|_p\\
 &\leq C_{m,d}\|WF\|_p+C_{m,d}(p')^{m-2}\|f\|_p.
\end{aligned}
\]
This is \eqref{eq:maximal-to-square}.
\end{proof}

\section{Square-function estimates for the representation}\label{sec:square-estimates}

We remain in the finite setting of Section~\ref{sec:geometry}, with
the same grids $\cD^i$, cubes $I_0^i$, smallest side length $2^{-k}$,
and finite collections $\cD_0^i$ and $(\cD_0^i)^\circ$.

We first deduce the finite maximal estimate from the weighted
square-function estimate in Lemma~\ref{lem:projected-square},
which will be proved at the end of this section.

\begin{lemma}[Finite incomparable maximal estimate]\label{lem:represented-bound}
Let $m\geq2$ and let $\mathcal G\subset\cD$ be finite and incomparable.
Suppose every $I\in\mathcal G$ satisfies $I\subset I_0$ and
$\ell(I^i)\geq2^{-k}$ for all $i$, where $I_0\in\cD$ and $k\in\Z$.
Let $f$ be supported on $I_0$, strictly positive there, and constant
on every cube $J\in\cD$ with $J\subset I_0$ and $\ell(J)=2^{-k}$.
Then, for $1<p\leq3/2$,
\begin{equation}\label{eq:represented-bound}
 \|M_{\mathcal G}f\|_p\leq C_{m,d}(p')^{m-1}\|f\|_p,
\end{equation}
where $C_{m,d}$ depends only on $m$ and $d$.
\end{lemma}

\begin{proof}
The reduction \eqref{eq:maximal-to-square} leads to estimating $\|WF\|_p$;
the other term already has the acceptable power
$(p')^{m-2}\leq(p')^{m-1}$.
We will prove in Lemma~\ref{lem:projected-square} that
\[
 \int_{I_0}\frac{(WF)^2}{(M_{\mathcal B}f)^{2-p}}
 \leq C_{m,d}(p-1)^{-(m-1)}\|f\|_p^p.
\]
The denominator involves $M_{\mathcal B}f$ because it controls
the cross-coordinate averages in \eqref{eq:averaged-denominator}.
To use this weighted estimate, write
\[
 (WF)^p=\left(\frac{(WF)^2}{(M_{\mathcal B}f)^{2-p}}\right)^{p/2}
         (M_{\mathcal B}f)^{p(2-p)/2}.
\]
H\"older's inequality with exponents $2/p$ and $2/(2-p)$ gives
\begin{align}
 \|WF\|_p
 &=\left(\int_{I_0}
        \left(\frac{(WF)^2}{(M_{\mathcal B}f)^{2-p}}\right)^{p/2}
             (M_{\mathcal B}f)^{p(2-p)/2}\right)^{1/p}\notag\\
 &\leq\left(\int_{I_0}\frac{(WF)^2}{(M_{\mathcal B}f)^{2-p}}\right)^{1/2}
       \left(\int_{I_0}(M_{\mathcal B}f)^p\right)^{(2-p)/(2p)}\notag\\
 &=\left(\int_{I_0}\frac{(WF)^2}{(M_{\mathcal B}f)^{2-p}}\right)^{1/2}
       \|M_{\mathcal B}f\|_p^{1-p/2}\notag\\
 &\leq C_{m,d}(p-1)^{-(m-1)/2}
         \|f\|_p^{p/2}\|M_{\mathcal B}f\|_p^{1-p/2}.
         \label{eq:holder-absorption}
\end{align}
This estimate brings in $\|M_{\mathcal B}f\|_p$ even if our initial
target is $M_{\mathcal G}f$.  We therefore use
\eqref{eq:maximal-to-square} with $M_{\mathcal B}f$ on the left:
the same norm then appears on both sides, allowing absorption.
Combining \eqref{eq:maximal-to-square} and \eqref{eq:holder-absorption}
gives, for a constant $C_{m,d}\geq1$ depending only on $m$ and $d$,
\[
\begin{aligned}
 \|M_{\mathcal B}f\|_p
 &\leq C_{m,d}(p')^{m-2}\|f\|_p\\
 &\quad+C_{m,d}(p-1)^{-(m-1)/2}
       \|f\|_p^{p/2}\|M_{\mathcal B}f\|_p^{1-p/2}.
\end{aligned}
\]
We handle the additive term before dividing.  If
$C_{m,d}(p')^{m-2}\|f\|_p\geq\tfrac12\|M_{\mathcal B}f\|_p$, then
\[
 \|M_{\mathcal B}f\|_p\leq2C_{m,d}(p')^{m-2}\|f\|_p,
\]
which is already stronger than the desired estimate.
Otherwise, the combined inequality gives
\[
 \|M_{\mathcal B}f\|_p
 \leq\tfrac12\|M_{\mathcal B}f\|_p
      +C_{m,d}(p-1)^{-(m-1)/2}
       \|f\|_p^{p/2}\|M_{\mathcal B}f\|_p^{1-p/2}.
\]
Moving $\tfrac12\|M_{\mathcal B}f\|_p$ to the left gives
\[
 \tfrac12\|M_{\mathcal B}f\|_p
 \leq C_{m,d}(p-1)^{-(m-1)/2}
       \|f\|_p^{p/2}\|M_{\mathcal B}f\|_p^{1-p/2}.
\]
Since $\|M_{\mathcal B}f\|_p>0$, division by $\|M_{\mathcal B}f\|_p^{1-p/2}$ and then
raising to the power $2/p$ gives
\[
 \|M_{\mathcal B}f\|_p
 \leq C_{m,d}^{2/p}(p-1)^{-(m-1)/p}\|f\|_p
 \leq C_{m,d}(p-1)^{-(m-1)/p}\|f\|_p,
\]
where we used $2/p\leq2$ in the last inequality.  Finally,
$1<p\leq3/2$ gives $0<p-1<1$.  Since $(m-1)/p\leq m-1$,
\[
 (p-1)^{m-1}\leq(p-1)^{(m-1)/p}.
\]
Taking reciprocals and using $(p-1)^{-1}\leq p'$ gives
\[
 (p-1)^{-(m-1)/p}\leq(p-1)^{-(m-1)}\leq(p')^{m-1}.
\]
Also $(p')^{m-2}\leq(p')^{m-1}$.
Both cases therefore yield, after enlarging the constant,
\[
 \|M_{\mathcal G}f\|_p\leq\|M_{\mathcal B}f\|_p
 \leq C_{m,d}(p')^{m-1}\|f\|_p,
\]
as claimed.
\end{proof}

It remains to prove the weighted estimate in
Lemma~\ref{lem:projected-square}.  We first establish a one-coordinate
estimate, iterate it to obtain a product estimate for $f$, and then
transfer that estimate to $F$ using the cross-coordinate averaging
identity \eqref{eq:cross-projection}.

\subsection{A one-coordinate estimate for two functions}

We first prove a scalar inequality for
$\Psi_p(a,b)=a^2/b^{2-p}$, defined for $a\in\R$ and $b>0$.

\begin{lemma}[A scalar estimate]\label{lem:scalar-estimate}
Let $1<p\leq2$, $a,h,\eta\in\R$, and $b>0$ with $b+\eta>0$.
Then
\begin{equation}\label{eq:scalar-estimate}
\begin{aligned}
 &\Psi_p(a+h,b+\eta)-\Psi_p(a,b)
       -\frac{2ah}{b^{2-p}}+\frac{(2-p)a^2\eta}{b^{3-p}}\\
 &\qquad\geq\frac{p-1}{3-p}
       \frac{h^2}{\max\{b,b+\eta\}^{2-p}}.
\end{aligned}
\end{equation}
\end{lemma}

\begin{proof}
Put $r=2-p\in[0,1)$ and set
\[
 a_s=a+sh,\qquad b_s=b+s\eta,\qquad 0\leq s\leq1.
\]
We regard $a_s$ and $b_s$ as functions of $s$; primes denote
differentiation with respect to $s$.
Since $b>0$ and $b+\eta>0$,
\[
 0<b_s=(1-s)b+s(b+\eta)\leq\max\{b,b+\eta\}.
\]
Write $q(s)=\Psi_p(a_s,b_s)=a_s^2b_s^{-r}$.
Since $a_s'=h$ and $b_s'=\eta$,
\[
 \frac{\ud}{\ud s}b_s^{-r}=-r\eta b_s^{-r-1},
 \qquad
 \frac{\ud}{\ud s}b_s^{-r-1}=-(r+1)\eta b_s^{-r-2}.
\]
The two derivatives of $q$ are therefore
\begin{align*}
 q'(s)&=2a_shb_s^{-r}-ra_s^2\eta b_s^{-r-1},\\
 q''(s)
 &=2h^2b_s^{-r}-2ra_sh\eta b_s^{-r-1}\\
 &\qquad-r\bigl(2a_sh\eta b_s^{-r-1}
                    -(r+1)a_s^2\eta^2b_s^{-r-2}\bigr)\\
 &=2b_s^{-r}h^2-4ra_sb_s^{-r-1}h\eta
       +r(r+1)a_s^2b_s^{-r-2}\eta^2\\
 &=b_s^{-r-2}\Bigl[
       r(r+1)(a_s\eta)^2-4r(a_s\eta)(b_sh)+2(b_sh)^2\Bigr].
\end{align*}
The bracket is a quadratic expression in $a_s\eta$ and $b_sh$.
Completing its square gives
\[
\begin{aligned}
 &r(r+1)(a_s\eta)^2-4r(a_s\eta)(b_sh)+2(b_sh)^2\\
 &\qquad=r(r+1)\Bigl(a_s\eta-\frac{2b_sh}{r+1}\Bigr)^2
       +\Bigl(2-\frac{4r}{r+1}\Bigr)(b_sh)^2.
\end{aligned}
\]
The square contributes the first two terms and
$\frac{4r}{r+1}(b_sh)^2$; the remaining term restores the coefficient
of $(b_sh)^2$ to $2$.  Since
$2-\frac{4r}{r+1}=\frac{2(1-r)}{1+r}$, multiplying by $b_s^{-r-2}$ gives
\begin{equation}\label{eq:convexity-calculation}
\begin{aligned}
 q''(s)
 &=r(r+1)b_s^{-r-2}
       \left(a_s\eta-\frac{2b_sh}{r+1}\right)^2\\
 &\qquad+\frac{2(1-r)}{1+r}b_s^{-r}h^2\\
 &\geq\frac{2(1-r)}{1+r}b_s^{-r}h^2
 =\frac{2(p-1)}{3-p}\,\frac{h^2}{b_s^{2-p}}.
\end{aligned}
\end{equation}
Here we simply used that the squared term is nonnegative.

Apply the fundamental theorem of calculus first to $q$ and then
to $q'$.  Since $b_t\leq\max\{b,b+\eta\}$ and
$\int_0^1\int_0^s\ud t\ud s=\int_0^1s\ud s=1/2$, we obtain
\begin{align*}
 q(1)-q(0)-q'(0)
 &=\int_0^1\bigl(q'(s)-q'(0)\bigr)\ud s\\
 &=\int_0^1\int_0^s q''(t)\ud t\ud s\\
 &\geq\frac{2(p-1)}{3-p}
       \frac{h^2}{\max\{b,b+\eta\}^{2-p}}
       \int_0^1\int_0^s\ud t\ud s\\
 &=\frac{p-1}{3-p}\frac{h^2}{\max\{b,b+\eta\}^{2-p}}.
\end{align*}
Finally,
\[
\begin{gathered}
 q(1)=\Psi_p(a+h,b+\eta),\qquad q(0)=\Psi_p(a,b),\\
 q'(0)=\frac{2ah}{b^{2-p}}-\frac{(2-p)a^2\eta}{b^{3-p}},
\end{gathered}
\]
so the left side is exactly that of \eqref{eq:scalar-estimate}.
\end{proof}

At $p=2$, the following lemma is the usual $L^2$ estimate for
dyadic martingale differences, with constant $1$.  For $1<p<2$,
it gives a weighted version of that estimate, with a loss of order
$(p-1)^{-1}$ as $p$ approaches $1$.

\begin{lemma}[A weighted square-function estimate]\label{lem:weighted-square}
Fix $i$.  Let $u,v$ be functions on $I_0^i$, with $v>0$, both
constant on every $Q\in\cD_0^i$ with $\ell(Q)=2^{-k}$.  Then
\begin{equation}\label{eq:weighted-square}
 \sum_{Q\in(\cD_0^i)^\circ}\int_Q
 \frac{|\Delta_Qu|^2}{\langle v\rangle_Q^{2-p}}
 \leq2^{d_i(2-p)}\frac{3-p}{p-1}\int_{I_0^i}\frac{|u|^2}{v^{2-p}},
 \qquad1<p\leq2.
\end{equation}
\end{lemma}

\begin{proof}
We apply Lemma~\ref{lem:scalar-estimate} to parent and child averages.
For $Q\in(\cD_0^i)^\circ$ and $J\in\ch(Q)$, take
\[
\begin{aligned}
 (a_0,b_0)&=(\langle u\rangle_Q,\langle v\rangle_Q),
 & (a_1,b_1)&=(\langle u\rangle_J,\langle v\rangle_J),\\
 h&=a_1-a_0,
 & \eta&=b_1-b_0.
\end{aligned}
\]
Positivity of $v$ gives $b_0,b_1>0$ and
\[
 b_1=\frac1{|J|}\int_Jv
 \leq\frac1{|J|}\int_Qv
 =\frac{|Q|}{|J|}b_0=2^{d_i}b_0.
\]
Consequently,
\[
 \max\{b_0,b_1\}\leq2^{d_i}b_0.
\]
The scalar inequality therefore gives
\begin{equation}\label{eq:pair-increment}
\begin{aligned}
 &\Psi_p(a_1,b_1)-\Psi_p(a_0,b_0)
       -\frac{2a_0}{b_0^{2-p}}(a_1-a_0)
       +\frac{(2-p)a_0^2}{b_0^{3-p}}(b_1-b_0)\\
 &\qquad\geq\frac{p-1}{3-p}
       \frac{|a_1-a_0|^2}{\max\{b_0,b_1\}^{2-p}}\\
 &\qquad\geq2^{-d_i(2-p)}\frac{p-1}{3-p}
       \frac{|\langle u\rangle_J-\langle u\rangle_Q|^2}
            {\langle v\rangle_Q^{2-p}}.
\end{aligned}
\end{equation}

Multiply \eqref{eq:pair-increment} by $|J|$ and sum over the children.
The coefficients of $a_1-a_0$ and $b_1-b_0$ depend only on the
parent averages.  The sums of these two differences are
\begin{align*}
 \sum_{J\in\ch(Q)}|J|(\langle u\rangle_J-\langle u\rangle_Q)
 &=\int_Qu-|Q|\langle u\rangle_Q=0,\\
 \sum_{J\in\ch(Q)}|J|(\langle v\rangle_J-\langle v\rangle_Q)
 &=\int_Qv-|Q|\langle v\rangle_Q=0.
\end{align*}
Thus the linear terms cancel.  Moreover,
\[
 \Delta_Qu(x)=\langle u\rangle_J-\langle u\rangle_Q
 \quad(x\in J\in\ch(Q)),
\]
so disjointness of the children gives
\[
 \sum_{J\in\ch(Q)}|J|
 \frac{|\langle u\rangle_J-\langle u\rangle_Q|^2}
      {\langle v\rangle_Q^{2-p}}
 =\int_Q\frac{|\Delta_Qu|^2}{\langle v\rangle_Q^{2-p}}.
\]
The sum of \eqref{eq:pair-increment} is therefore
\begin{align*}
 2^{-d_i(2-p)}\frac{p-1}{3-p}\int_Q
       \frac{|\Delta_Qu|^2}{\langle v\rangle_Q^{2-p}}
 &\leq\sum_{J\in\ch(Q)}|J|
          \Psi_p(\langle u\rangle_J,\langle v\rangle_J)\\
 &\hspace{8mm}-|Q|\Psi_p(\langle u\rangle_Q,\langle v\rangle_Q).
\end{align*}
Summing over $Q\in(\cD_0^i)^\circ$ gives
\begin{align*}
 &\sum_Q\left[
     \sum_{J\in\ch(Q)}|J|\Psi_p(\langle u\rangle_J,\langle v\rangle_J)
          -|Q|\Psi_p(\langle u\rangle_Q,\langle v\rangle_Q)\right]\\
 &=\sum_{\substack{J\in\cD_0^i\\J\ne I_0^i}}
       |J|\Psi_p(\langle u\rangle_J,\langle v\rangle_J)
       -\sum_{Q\in(\cD_0^i)^\circ}
       |Q|\Psi_p(\langle u\rangle_Q,\langle v\rangle_Q)\\
 &=\sum_{\substack{J\in\cD_0^i\\\ell(J)=2^{-k}}}
       |J|\Psi_p(\langle u\rangle_J,\langle v\rangle_J)
       -|I_0^i|\Psi_p(\langle u\rangle_{I_0^i},\langle v\rangle_{I_0^i})\\
 &=\int_{I_0^i}\Psi_p(u,v)
       -|I_0^i|\Psi_p(\langle u\rangle_{I_0^i},\langle v\rangle_{I_0^i})
 \leq\int_{I_0^i}\Psi_p(u,v).
\end{align*}
The last equality uses that $u$ and $v$ are constant on every
$J\in\cD_0^i$ with $\ell(J)=2^{-k}$, so
$|J|\Psi_p(\langle u\rangle_J,\langle v\rangle_J)
=\int_J\Psi_p(u,v)$.
The final inequality uses $\Psi_p\geq0$.
This proves \eqref{eq:weighted-square}.
\end{proof}

\subsection{The product square-function estimate}

We return to the positive function $f$ on $I_0$ from
Section~\ref{sec:geometry}, constant on the smallest cubes.
Our next step is to iterate Lemma~\ref{lem:weighted-square} over
a set of coordinates $A$, obtaining a weighted estimate for the
products $\Delta_{K^A}f$.  In the next subsection, cross-coordinate
averaging will transfer this estimate to the projection $F$.

We write $A^c=\{1,\ldots,m\}\setminus A$.

\begin{lemma}[A product square-function estimate]\label{lem:product-square}
For every nonempty coordinate set $A$,
\begin{equation}\label{eq:product-square}
 \sum_{K^A\in(\cD_0^A)^\circ}\int_{K^A\times I_0^{A^c}}
     \frac{|\Delta_{K^A}f|^2}{(E_{K^A}f)^{2-p}}
 \leq2^{(2-p)\sum_{i\in A}d_i}
       \left(\frac{3-p}{p-1}\right)^{|A|}\|f\|_p^p,
 \qquad1<p\leq2.
\end{equation}
\end{lemma}

\begin{proof}
We first separate one coordinate $j\in A$ in the left side of
\eqref{eq:product-square}.  Put $B=A\setminus\{j\}$; each sum
over $K^i$ below is over $(\cD_0^i)^\circ$.
For each fixed $K^B$, write
\[
 u_B=\Delta_{K^B}f,
 \qquad
 v_B=E_{K^B}f.
\]
Since operators in distinct coordinates commute, Fubini's theorem gives
\[
\begin{aligned}
 &\sum_{K^A}\int_{K^A\times I_0^{A^c}}
       \frac{|\Delta_{K^A}f(x)|^2}{(E_{K^A}f(x))^{2-p}}\ud x\\
 &=\sum_{K^B}\int_{K^B\times I_0^{A^c}}
       \Bigg[\sum_{K^j}\int_{K^j}
       \frac{|\Delta_{K^j}u_B(x)|^2}{(E_{K^j}v_B(x))^{2-p}}\ud x_j\Bigg]
       \ud x^{A_j}.
\end{aligned}
\]
Since $j\notin B$, the operators defining $u_B$ and $v_B$ act only
in other coordinates and therefore preserve $f$'s constancy on the
smallest cubes in coordinate $j$.  For fixed
$x^{A_j}\in K^B\times I_0^{A^c}$, we also have $v_B(x)>0$
for every $x_j\in I_0^j$.
Lemma~\ref{lem:weighted-square} therefore estimates the inner sum by
\[
 \sum_{K^j}\int_{K^j}
       \frac{|\Delta_{K^j}u_B(x)|^2}{(E_{K^j}v_B(x))^{2-p}}\ud x_j
 \leq2^{d_j(2-p)}\frac{3-p}{p-1}
       \int_{I_0^j}\frac{u_B(x)^2}{v_B(x)^{2-p}}\ud x_j.
\]
Substituting this bound into the preceding equality and using
$B^c=A^c\cup\{j\}$ gives
\begin{align}
 &\sum_{K^A}\int_{K^A\times I_0^{A^c}}
       \frac{|\Delta_{K^A}f|^2}{(E_{K^A}f)^{2-p}}
 \notag\\
 &\qquad\leq2^{d_j(2-p)}\frac{3-p}{p-1}
       \sum_{K^B}\int_{K^B\times I_0^{B^c}}
       \frac{|\Delta_{K^B}f|^2}{(E_{K^B}f)^{2-p}}.
 \label{eq:square-iteration}
\end{align}
Iterating \eqref{eq:square-iteration} $|A|$ times bounds the left
side of \eqref{eq:product-square} by
\[
 \left(\prod_{i\in A}2^{d_i(2-p)}\frac{3-p}{p-1}\right)
 \int_{I_0}\frac{f^2}{f^{2-p}}
 =2^{(2-p)\sum_{i\in A}d_i}
       \left(\frac{3-p}{p-1}\right)^{|A|}\|f\|_p^p.
\]
\end{proof}

\subsection{Square-function estimates for the projected function}

Recall from Definition~\ref{def:projection} that
\[
 F=f-\sum_{\substack{L\in\cD_0^\circ\\
             L\subset I\text{ for some }I\in\mathcal G}}\Delta_Lf.
\]
The property of $F$ used here is the cross-coordinate averaging
identity from Lemma~\ref{lem:cross}: for
$A_j=\{1,\ldots,m\}\setminus\{j\}$,
\[
 \Delta_{K^{A_j}}F=X_{K^{A_j}}\Delta_{K^{A_j}}f,
 \qquad K^{A_j}\in(\cD_0^{A_j})^\circ.
\]
We will use this identity to transfer the estimate of
Lemma~\ref{lem:product-square}, with $|A|=m-1$, from $f$ to $F$.

We use the square functions $S_{A_j}F$ and $WF$ defined in
\eqref{eq:projected-square-functions}.  The maximal function
$M_{\mathcal B}f$ bounds the averages that appear in the
denominators after cross-coordinate averaging.

\begin{lemma}[A square-function estimate for the projected function]\label{lem:projected-square}
We have
\begin{equation}\label{eq:projected-square}
 \int_{I_0}\frac{(WF)^2}{(M_{\mathcal B}f)^{2-p}}
 \leq C_{m,d}(p-1)^{-(m-1)}\|f\|_p^p,
 \qquad1<p\leq2.
\end{equation}
The constant depends only on $m$ and $d$.
\end{lemma}

\begin{proof}
We expand the left side of \eqref{eq:projected-square}.
Since $\Delta_{K^{A_j}}F(x)=0$ unless $x^{A_j}\in K^{A_j}$,
the definition of $WF$ and Lemma~\ref{lem:cross} give
\begin{equation}\label{eq:projected-square-expansion}
\begin{aligned}
 \int_{I_0}\frac{(WF)^2}{(M_{\mathcal B}f)^{2-p}}
 &=\sum_{j=1}^m\sum_{K^{A_j}\in(\cD_0^{A_j})^\circ}
       \int_{K^{A_j}\times I_0^j}
       \frac{|\Delta_{K^{A_j}}F|^2}{(M_{\mathcal B}f)^{2-p}}\\
 &=\sum_{j=1}^m\sum_{K^{A_j}\in(\cD_0^{A_j})^\circ}
       \int_{K^{A_j}\times I_0^j}
       \frac{|X_{K^{A_j}}\Delta_{K^{A_j}}f|^2}{(M_{\mathcal B}f)^{2-p}}.
\end{aligned}
\end{equation}
Fix $j$ and $K^{A_j}\in(\cD_0^{A_j})^\circ$.
By \eqref{eq:averaged-denominator},
\begin{equation}\label{eq:averaged-denominator-bound}
 X_{K^{A_j}}E_{K^{A_j}}f(x)\leq M_{\mathcal B}f(x),
 \qquad x\in K^{A_j}\times I_0^j.
\end{equation}
Fix $x\in K^{A_j}\times I_0^j$, and let
$P^j\in\mathcal P_{K^{A_j}}^j$ contain $x_j$.
Put $u=\Delta_{K^{A_j}}f$ and $v=E_{K^{A_j}}f$;
then $v>0$ on $K^{A_j}\times I_0^j$.
Cauchy--Schwarz for the average over $P^j$ gives
\[
\begin{aligned}
 |E_{P^j}u(x)|^2
 &=\left|E_{P^j}\!\left(\frac{u}{v^{(2-p)/2}}v^{(2-p)/2}
                  \right)(x)\right|^2\\
 &\leq E_{P^j}\!\left(\frac{u^2}{v^{2-p}}\right)(x)
       E_{P^j}(v^{2-p})(x).
\end{aligned}
\]
Monotonicity of normalized $L^s$ averages
and division by $(E_{P^j}v(x))^{2-p}$ give
\[
\begin{aligned}
 E_{P^j}(v^{2-p})(x)&\leq(E_{P^j}v(x))^{2-p},\\
 \frac{|E_{P^j}u(x)|^2}{(E_{P^j}v(x))^{2-p}}
 &\leq E_{P^j}\!\left(\frac{u^2}{v^{2-p}}\right)(x).
\end{aligned}
\]
Since $x_j\in P^j$, we have
$X_{K^{A_j}}g(x)=E_{P^j}g(x)$.
Thus \eqref{eq:averaged-denominator-bound} and $2-p\geq0$ give
\begin{equation}\label{eq:averaged-square}
\begin{aligned}
 \frac{|X_{K^{A_j}}\Delta_{K^{A_j}}f(x)|^2}{(M_{\mathcal B}f(x))^{2-p}}
 &\leq\frac{|X_{K^{A_j}}\Delta_{K^{A_j}}f(x)|^2}
                  {(X_{K^{A_j}}E_{K^{A_j}}f(x))^{2-p}}\\
 &\leq X_{K^{A_j}}\!\left(
       \frac{|\Delta_{K^{A_j}}f|^2}{(E_{K^{A_j}}f)^{2-p}}\right)(x).
\end{aligned}
\end{equation}
For fixed $x^{A_j}\in K^{A_j}$, the definition of
$X_{K^{A_j}}$ gives
\[
\begin{aligned}
 \int_{I_0^j}X_{K^{A_j}}g(x)\ud x_j
 &=\sum_{P^j\in\mathcal P_{K^{A_j}}^j}
       \int_{P^j}E_{P^j}g(x)\ud x_j\\
 &=\sum_{P^j\in\mathcal P_{K^{A_j}}^j}
       \int_{P^j}g(x)\ud x_j
 =\int_{I_0^j}g(x)\ud x_j.
\end{aligned}
\]
Integrating \eqref{eq:averaged-square} over
$K^{A_j}\times I_0^j$, summing over $K^{A_j}$, and applying
Lemma~\ref{lem:product-square} gives
\begin{align*}
 &\sum_{K^{A_j}\in(\cD_0^{A_j})^\circ}\int_{K^{A_j}\times I_0^j}
       \frac{|X_{K^{A_j}}\Delta_{K^{A_j}}f|^2}{(M_{\mathcal B}f)^{2-p}}\\
 &\leq\sum_{K^{A_j}\in(\cD_0^{A_j})^\circ}\int_{K^{A_j}\times I_0^j}
       X_{K^{A_j}}\!\left(\frac{|\Delta_{K^{A_j}}f|^2}
                             {(E_{K^{A_j}}f)^{2-p}}\right)\\
 &=\sum_{K^{A_j}\in(\cD_0^{A_j})^\circ}\int_{K^{A_j}\times I_0^j}
       \frac{|\Delta_{K^{A_j}}f|^2}{(E_{K^{A_j}}f)^{2-p}}\\
 &\leq2^{(2-p)\sum_{i\in A_j}d_i}
       \left(\frac{3-p}{p-1}\right)^{m-1}\|f\|_p^p.
\end{align*}
Summing these bounds in \eqref{eq:projected-square-expansion}, and
using $\sum_{i\in A_j}d_i\leq d$, $2-p\leq1$, and $3-p\leq2$, gives
\[
\begin{aligned}
 \int_{I_0}\frac{(WF)^2}{(M_{\mathcal B}f)^{2-p}}
 &\leq m2^d\left(\frac{2}{p-1}\right)^{m-1}\|f\|_p^p
 =m2^{d+m-1}(p-1)^{-(m-1)}\|f\|_p^p,
\end{aligned}
\]
which proves \eqref{eq:projected-square}.
\end{proof}

\section{The Zygmund estimate for arbitrary side lengths}
\label{sec:continuous-zygmund}

We now remove the dyadic restriction from the Zygmund estimate.
The coordinate factors remain $\R^{d_1},\ldots,\R^{d_m}$, but the
rectangles may have arbitrary positions and
positive side lengths.

Let $m\geq3$ and let
$\phi\colon(0,\infty)^{m-1}\to(0,\infty)$ be coordinatewise
nondecreasing. Let $\mathcal R_\phi$ consist of all rectangles
$R=R^1\times\cdots\times R^m$ with axis-parallel cube factors
$R^i\subset\R^{d_i}$ and side-length tuple
\[
 (s_1,\ldots,s_{m-1},\phi(s_1,\ldots,s_{m-1})),
 \qquad s_1,\ldots,s_{m-1}>0.
\]
For $f\in L^1_{\mathrm{loc}}(\R^d)$, define
\[
 M_{\mathcal R_\phi}f(x)
 =\sup_{\substack{R\in\mathcal R_\phi\\x\in R}}
       \langle|f|\rangle_R.
\]
This function is measurable: its strict level sets are open.
Indeed, $R+a\in\mathcal R_\phi$ whenever $R\in\mathcal R_\phi$, and
local translation continuity gives
\[
 \bigl|\langle|f|\rangle_{R+a}-\langle|f|\rangle_R\bigr|
 \leq\frac1{|R|}\int_R|f(x+a)-f(x)|\ud x\longrightarrow0
 \quad\text{as }a\to0.
\]
If $x\in R$ and $\langle|f|\rangle_R>\lambda$, a sufficiently small
translation therefore gives
\[
 x\in\operatorname{int}(R+a)
 \subset\{M_{\mathcal R_\phi}f>\lambda\}.
\]
This proves openness, including at points on the boundary of $R$.

\begin{theorem}[The continuous Zygmund estimate]\label{thm:continuous-zygmund}
For the family $\mathcal R_\phi$ defined above and every $\lambda>0$,
\begin{equation}\label{eq:continuous-zygmund}
 |\{x:M_{\mathcal R_\phi}f(x)>\lambda\}|
 \leq C_{m,d}\int_{\R^d}\frac{|f(x)|}{\lambda}
 \left[\log\!\left(e+\frac{|f(x)|}{\lambda}\right)\right]^{m-2}\ud x.
\end{equation}
The constant depends only on $m$ and $d$, not on $\phi$.
\end{theorem}

No continuity or doubling assumption on $\phi$ is imposed.
For $m=3$ with one-dimensional factors, this is C\'ordoba's
theorem~\cite{Cordoba}. The passage to arbitrary side lengths in
that case is also discussed by Mart\'inez~\cite{Martinez}; see
Rey's account in~\cite{Rey}.
We use Rey's selection and slicing argument after covering the
rectangles by dyadic rectangles of prescribed side lengths.
The slices of these covers need not be incomparable. The next theorem
extends our maximal and sparse-overlap estimates to the weaker
condition satisfied by these slices.

\subsection{A weaker dyadic containment hypothesis}

\begin{theorem}\label{thm:weaker-containment}
Let $m\geq2$ and suppose $\mathcal G\subset\cD$ satisfies
\begin{equation}\label{eq:weaker-containment}
 I,J\in\mathcal G,\quad I\subset J
 \quad\Longrightarrow\quad I^i=J^i
              \text{ for some }i\in\{1,\ldots,m\}.
\end{equation}
Then
\begin{equation}\label{eq:weaker-maximal}
 \|M_{\mathcal G}f\|_p
 \leq C_{m,d}(p')^{m-1}\|f\|_p,\qquad1<p\leq2.
\end{equation}
If $\mathcal G$ is also $\eta$-sparse and has finite shadow, the
moment and exponential estimates of Theorem~\ref{thm:overlap}
hold as well.
\end{theorem}

Condition~\eqref{eq:weaker-containment} allows containment when at
least one coordinate cube is unchanged; it excludes strict
containment in every coordinate. Thus Rey's conjectured overlap bound,
as well as the maximal estimate, extends beyond incomparable families.

\begin{proof}
By the reductions in Theorem~\ref{thm:maximal}, it suffices to use
the finite setting and positivity assumptions of
Lemma~\ref{lem:represented-bound}, with
\eqref{eq:weaker-containment} in place of incomparability and
$1<p\leq3/2$. Construct $F$, the partitions
$\mathcal P_{K^{A_j}}^j$, the averages $X_{K^{A_j}}$, and
$\mathcal B$ as in Section~\ref{sec:geometry}.

For any finite $\mathcal G$, the averaging rectangles satisfy
\[
 E_RF=E_Rf,\qquad R\in\mathcal B,
\]
and the representation in Lemma~\ref{lem:representation} holds
on $\mathcal B$. That proof uses only the definition of $F$ and
the maximality of the partition cubes, not incomparability.
The cross-coordinate identities and weighted square-function
estimate likewise do not require incomparability.

In the incomparable case, the argument was to use
\[
 M_{\mathcal G}f\leq M_{\mathcal B}f
\]
and then apply the representation and estimates on $\mathcal B$.
Its only essential use of incomparability was therefore the inclusion
$\mathcal G\subset\mathcal B$ from
Lemma~\ref{lem:original-averaging-rectangles}.

Under \eqref{eq:weaker-containment}, the inclusion
$\mathcal G\subset\mathcal B$ can fail.
We instead verify $E_IF=E_If$ directly on $\mathcal G$.
Once this identity holds on $\mathcal G\cup\mathcal B$, the
representation argument of Lemma~\ref{lem:representation} and the
reduction in Lemma~\ref{lem:maximal-to-square} apply to that union.
Only the estimate for $WF$ then remains, and the existing weighted
estimate gives it by the same absorption argument.

To verify $E_IF=E_If$, fix $I\in\mathcal G$ and a difference $\Delta_Lf$
removed from $f$ in \eqref{eq:projection}, so $L\subset J$ for
some $J\in\mathcal G$. By \eqref{eq:average-difference},
\[
 E_I\Delta_Lf\ne0
 \quad\Longrightarrow\quad
 I^i\subsetneq L^i\subset J^i\quad(1\leq i\leq m).
\]
This contradicts \eqref{eq:weaker-containment}. Together with the
preserved averages on $\mathcal B$, this proves
\begin{equation}\label{eq:weaker-averages}
 E_IF=E_If,\qquad I\in\mathcal G\cup\mathcal B.
\end{equation}
Since $\Delta_LF=0$ for $L\subset I\in\mathcal G$ is already
part of \eqref{eq:projection-properties}, the representation proof in
Lemma~\ref{lem:representation} goes through unchanged on
$\mathcal G\cup\mathcal B$. The proof of
Lemma~\ref{lem:maximal-to-square} therefore gives
\[
 \|M_{\mathcal G\cup\mathcal B}f\|_p
 \leq C_{m,d}(p')^{m-2}\|f\|_p+C_{m,d}\|WF\|_p.
\]
It remains to estimate $WF$. The cross-coordinate identity in
Lemma~\ref{lem:cross} and the bound \eqref{eq:averaged-denominator}
are unchanged, so the weighted estimate of
Lemma~\ref{lem:projected-square} gives
\[
 \int_{I_0}\frac{(WF)^2}{(M_{\mathcal B}f)^{2-p}}
 \leq C_{m,d}(p-1)^{-(m-1)}\|f\|_p^p.
\]
Applying H\"older as in Lemma~\ref{lem:represented-bound} and
substituting into the preceding maximal estimate gives
\[
\begin{aligned}
 \|M_{\mathcal G\cup\mathcal B}f\|_p
 &\leq C_{m,d}(p')^{m-2}\|f\|_p\\
 &\quad+C_{m,d}(p-1)^{-(m-1)/2}
 \|f\|_p^{p/2}\|M_{\mathcal B}f\|_p^{1-p/2}.
\end{aligned}
\]
Now $M_{\mathcal B}f\leq M_{\mathcal G\cup\mathcal B}f$ and
$1-p/2\geq0$ give
\[
 \|M_{\mathcal B}f\|_p^{1-p/2}
 \leq\|M_{\mathcal G\cup\mathcal B}f\|_p^{1-p/2}.
\]
With this bound, the preceding inequality is precisely the one
absorbed in Lemma~\ref{lem:represented-bound}. Hence
\[
 \|M_{\mathcal G}f\|_p
 \leq\|M_{\mathcal G\cup\mathcal B}f\|_p
 \leq C_{m,d}(p')^{m-1}\|f\|_p.
\]
The reductions in Theorem~\ref{thm:maximal} give
\eqref{eq:weaker-maximal} for general $\mathcal G$ and $1<p\leq2$.
The proof of Theorem~\ref{thm:overlap} then gives the sparse
conclusions, using \eqref{eq:weaker-maximal} in its duality estimate.
\end{proof}

\subsection{Dyadic covers with prescribed side lengths}

For $s>0$, define
\begin{equation}\label{eq:rounding-rule}
 \rho(s)=\lceil\log_2s\rceil+2,\qquad
 \rho(s)=k\ \Longleftrightarrow\ 2^{k-3}<s\leq2^{k-2}.
\end{equation}
In particular, $4s\leq2^{\rho(s)}<8s$.

We use the following standard shifted-grid covering fact, with
prescribed side length.

\begin{lemma}\label{lem:prescribed-cover}
For each $i$, there are $3^{d_i}$ dyadic cube grids
$\cD^{i,a}$, indexed by $a\in\{0,1/3,2/3\}^{d_i}$, such that
every axis-parallel cube $Q\subset\R^{d_i}$ of side length $s$
is contained in a cube $P$ from one of these grids with
$\ell(P)=2^{\rho(s)}$.
\end{lemma}

\begin{proof}
Use the grids
\[
 \cD^{i,a}
 =\{2^k([0,1)^{d_i}+z+(-1)^ka):
                     k\in\Z,\ z\in\Z^{d_i}\}.
\]
These are dyadic grids, since the shifts at two consecutive scales satisfy
\[
 2^{k+1}(-1)^{k+1}a-2^k(-1)^ka
 =-3\cdot2^k(-1)^ka\in2^k\Z^{d_i}.
\]
Fix $k=\rho(s)$ and $L=2^k$. In one scalar coordinate, the
intervals of length $L$ in the partition with shift
$t\in\{0,1/3,2/3\}$ are
\[
 L([0,1)+z+(-1)^kt),\qquad z\in\Z.
\]
Their endpoints form $L(\Z+(-1)^kt)$. As $t$ varies, these
three endpoint sets are, in some order,
\[
 L\Z,\qquad L\Z+\tfrac L3,\qquad L\Z+\tfrac{2L}3.
\]
They are disjoint and their union is $(L/3)\Z$, so any two
distinct endpoints are at least $L/3$ apart.
Each coordinate projection of $Q$ is an interval of length
$s\leq L/4<L/3$, hence contains at most one of these endpoints
altogether. At least two of the three partitions therefore have
no endpoint on that interval, so it is contained in one of their
intervals of length $L$.
Making this choice in each coordinate gives
$a\in\{0,1/3,2/3\}^{d_i}$ and $z\in\Z^{d_i}$ with
\[
 Q\subset P:=L([0,1)^{d_i}+z+(-1)^ka)\in\cD^{i,a},
 \qquad \ell(P)=L=2^{\rho(s)}.
\]
\end{proof}

Use all $3^d$ product grids
\[
 \cD_\tau=\cD^{1,\tau_1}\times\cdots\times\cD^{m,\tau_m},
 \qquad \tau_i\in\{0,1/3,2/3\}^{d_i}.
\]
The possible rounded side lengths are described by
\begin{equation}\label{eq:rounded-scales}
 \Gamma_\phi
 =\{(\rho(s_1),\ldots,\rho(s_{m-1}),\rho(\phi(s))):
                             s\in(0,\infty)^{m-1}\}.
\end{equation}
Let
\[
 \mathcal E_\phi^\tau
 =\{I\in\cD_\tau:
       (\log_2\ell(I^1),\ldots,\log_2\ell(I^m))\in\Gamma_\phi\}.
\]
Fix $R\in\mathcal R_\phi$ and write
$s_i=\ell(R^i)$ for $i<m$, so $\ell(R^m)=\phi(s)$.
Applying Lemma~\ref{lem:prescribed-cover} in each factor gives
\[
 R^i\subset I^i\in\cD^{i,\tau_i},\qquad
 \ell(I^i)=2^{\rho(\ell(R^i))},\qquad 1\leq i\leq m.
\]
Thus $R\subset I=I^1\times\cdots\times I^m\in\cD_\tau$ and
\[
 (\log_2\ell(I^1),\ldots,\log_2\ell(I^m))
 =(\rho(s_1),\ldots,\rho(s_{m-1}),\rho(\phi(s)))\in\Gamma_\phi.
\]
The prescribed side lengths therefore give
$I\in\mathcal E_\phi^\tau$. Since $2^{\rho(s)}<8s$,
\[
 \frac{|I|}{|R|}
 =\prod_{i=1}^m
       \left(\frac{2^{\rho(\ell(R^i))}}{\ell(R^i)}\right)^{d_i}
 <8^{d_1+\cdots+d_m}=8^d.
\]
Hence
\begin{equation}\label{eq:continuous-domination}
 M_{\mathcal R_\phi}f(x)
 \leq8^d\max_\tau M_{\mathcal E_\phi^\tau}f(x).
\end{equation}

The rounded family $\mathcal E_\phi^\tau$ need not be a dyadic
$\Phi$-Zygmund family: the same first $m-1$ rounded scales can
have several last-coordinate scales. Thus we cannot directly
invoke Rey's dyadic Zygmund reduction. We instead repeat the
selection and slicing argument using the following consequence
of monotonicity of $\phi$:
\begin{equation}\label{eq:rounded-order}
 k,l\in\Gamma_\phi,\quad k_i<l_i\ (1\leq i<m)
 \quad\Longrightarrow\quad k_m\leq l_m.
\end{equation}
To check this, choose $s,t$ giving $k,l$ in
\eqref{eq:rounded-scales}. For every $i<m$, \eqref{eq:rounding-rule}
gives
\[
 s_i\leq2^{k_i-2}\leq2^{l_i-3}<t_i.
\]
Thus $\phi(s)\leq\phi(t)$, and
$k_m=\rho(\phi(s))\leq\rho(\phi(t))=l_m$.

\subsection{Selection and the endpoint estimate}

We next show that the weak $L(\log L)^{m-2}$ estimate holds
uniformly for the families $\mathcal E_\phi^\tau$, as stated in
Theorem~\ref{thm:rounded-endpoint} below.
We use Rey's variational form of the C\'ordoba--Fefferman
selection~\cite{Rey}.

\begin{lemma}[Selection with comparable shadow]\label{lem:selection}
Let $\mathcal F$ be a finite family in a product of $k$ dyadic cube grids.
There is a $\mathcal G\subset\mathcal F$
such that
\begin{equation}\label{eq:selection-lemma}
 |\sh(\mathcal F)|\leq C_k|\sh(\mathcal G)|,
 \qquad
 |I\cap\sh(\mathcal G\setminus\{I\})|\leq\tfrac12|I|
 \quad(I\in\mathcal G).
\end{equation}
The constant depends only on the number $k$ of factors.
\end{lemma}

\begin{proof}
Put
\[
 \Lambda(\mathcal V)=2|\sh(\mathcal V)|-\sum_{I\in\mathcal V}|I|,
 \qquad\mathcal V\subset\mathcal F.
\]
Since $\mathcal F$ has finitely many subfamilies, $\Lambda$ attains
its maximum at some $\mathcal G\subset\mathcal F$.
For $I\in\mathcal G$, comparison with $\mathcal G\setminus\{I\}$ gives
\begin{align*}
 0&\leq\Lambda(\mathcal G)-\Lambda(\mathcal G\setminus\{I\})\\
  &=2\bigl(|\sh(\mathcal G)|-|\sh(\mathcal G\setminus\{I\})|\bigr)-|I|\\
  &=2|I\setminus\sh(\mathcal G\setminus\{I\})|-|I|.
\end{align*}
Thus $|I\setminus\sh(\mathcal G\setminus\{I\})|\geq|I|/2$,
which is the required overlap condition.
For $J\in\mathcal F\setminus\mathcal G$, comparison with
$\mathcal G\cup\{J\}$ gives
\begin{align*}
 0&\geq\Lambda(\mathcal G\cup\{J\})-\Lambda(\mathcal G)\\
  &=2\bigl(|\sh(\mathcal G\cup\{J\})|-|\sh(\mathcal G)|\bigr)-|J|\\
  &=2|J\setminus\sh(\mathcal G)|-|J|.
\end{align*}
Consequently,
\[
 |J\cap\sh(\mathcal G)|
 =|J|-|J\setminus\sh(\mathcal G)|\geq\tfrac12|J|.
\]
For $J\in\mathcal G$, we have
$\langle\1_{\sh(\mathcal G)}\rangle_J=1$.
Hence $\langle\1_{\sh(\mathcal G)}\rangle_J\geq1/2$ for every
$J\in\mathcal F$, so
\[
 \sh(\mathcal F)\subset
 \{M_{\cD}\1_{\sh(\mathcal G)}\geq\tfrac12\},
 \qquad
 |\sh(\mathcal F)|
 \leq4\|M_{\cD}\1_{\sh(\mathcal G)}\|_2^2
 \leq4^{k+1}|\sh(\mathcal G)|.
\]
Here $\cD$ is the full product of the $k$ grids and the last step
is \eqref{eq:ordinary-maximal}.  Thus one may take
$C_k=4^{k+1}$.
\end{proof}

\begin{theorem}\label{thm:rounded-endpoint}
Let $m\geq3$ and let
$\phi\colon(0,\infty)^{m-1}\to(0,\infty)$ be coordinatewise
nondecreasing. For each of the families $\mathcal E_\phi^\tau$
defined above and every $\lambda>0$,
\begin{equation}\label{eq:rounded-endpoint}
 |\{M_{\mathcal E_\phi^\tau}f>\lambda\}|
 \leq C_{m,d}\int_{\R^d}\frac{|f|}{\lambda}
       \left[\log\!\left(e+\frac{|f|}{\lambda}\right)\right]^{m-2}.
\end{equation}
The constant is independent of $\phi$ and $\tau$.
\end{theorem}

\begin{proof}
Fix $\tau$. Since $M_{\mathcal E_\phi^\tau}f
=M_{\mathcal E_\phi^\tau}|f|$, we may take $f\geq0$.
Homogeneity reduces the estimate to $\lambda=1$.
We may assume the integral on the right of
\eqref{eq:rounded-endpoint} is finite; in particular, $f\in L^1$.
Fix a finite family
\[
 \mathcal F\subset\{I\in\mathcal E_\phi^\tau:\langle f\rangle_I>1\},
 \qquad
 \Omega=\sh(\mathcal F).
\]
Lemma~\ref{lem:selection} gives $\mathcal G\subset\mathcal F$ with
\begin{equation}\label{eq:selection}
 |\Omega|\leq C_1|\sh(\mathcal G)|,
 \qquad
 |I\cap\sh(\mathcal G\setminus\{I\})|\leq\tfrac12|I|
 \quad(I\in\mathcal G),
\end{equation}
where $C_1$ depends only on $m$. Since $\langle f\rangle_I>1$
for $I\in\mathcal G$,
\begin{equation}\label{eq:levelset-pairing}
 |\Omega|\leq C_1\sum_{I\in\mathcal G}|I|
       \leq C_1\int_{\sh(\mathcal G)}f h_{\mathcal G}.
\end{equation}
We next prove exponential integrability of $h_{\mathcal G}^{1/(m-2)}$
and then use it to estimate the last integral.

First, \eqref{eq:selection} makes $\mathcal G$ incomparable:
having distinct $I,J\in\mathcal G$ with $I\subset J$ would imply
\[
 |I\cap\sh(\mathcal G\setminus\{I\})|=|I|>\tfrac12|I|.
\]
Fix $t\in\R^{d_m}$ and define the sliced family
\[
 \mathcal G_t=\{I^1\times\cdots\times I^{m-1}:
                      I\in\mathcal G,\ t\in I^m\}.
\]
No two distinct rectangles give the same member of $\mathcal G_t$.
Indeed, for $I,J\in\mathcal G$,
\[
\begin{gathered}
 I^i=J^i\ (i<m),\qquad t\in I^m\cap J^m\\
 \Longrightarrow\quad I^m\subset J^m\ \text{or}\ J^m\subset I^m\\
 \Longrightarrow\quad I\subset J\ \text{or}\ J\subset I
 \quad\Longrightarrow\quad I=J.
\end{gathered}
\]
The first implication uses dyadicity and the last incomparability.
Consequently, for $x\in\R^{d_1}\times\cdots\times\R^{d_{m-1}}$,
\begin{equation}\label{eq:sliced-overlap}
 h_{\mathcal G}(x,t)
 =\sum_{\substack{I\in\mathcal G\\t\in I^m}}
       \prod_{i=1}^{m-1}\1_{I^i}(x_i)
 =h_{\mathcal G_t}(x).
\end{equation}

In Rey's dyadic $\Phi$-Zygmund setting, the slices are incomparable:
if $I^i\subset J^i$ for every $i<m$, monotonicity of $\Phi$
gives $\ell(I^m)\leq\ell(J^m)$. The common point $t$ then gives
$I^m\subset J^m$, contradicting incomparability of $\mathcal G$
when $I\ne J$. For our rounded family $\mathcal E_\phi^\tau$,
\eqref{eq:rounded-order} gives this implication only when
$I^i\subsetneq J^i$ for every $i<m$. Thus the slices need not be
incomparable; instead, they satisfy \eqref{eq:weaker-containment}
in $m-1$ coordinates, which is why we use
Theorem~\ref{thm:weaker-containment}.
Suppose, to the contrary, that $I,J\in\mathcal G$ have
$t\in I^m\cap J^m$ and
$I^i\subsetneq J^i$ for every $i<m$.
Put $k_i=\log_2\ell(I^i)$ and $l_i=\log_2\ell(J^i)$.
Then $k,l\in\Gamma_\phi$ and $k_i<l_i$ for $i<m$, so
\eqref{eq:rounded-order} gives
\[
 k_m\leq l_m
 \quad\Longrightarrow\quad\ell(I^m)\leq\ell(J^m)
 \quad\Longrightarrow\quad I^m\subset J^m.
\]
The last implication uses dyadicity and the common point $t$.
Thus $I\subset J$, contradicting incomparability of $\mathcal G$.

The rest of the proof follows Rey's argument, using
Theorem~\ref{thm:weaker-containment} in place of the exponential
estimate for incomparable slices.

To prove uniform sparseness of the slices, list the rectangles $I\in\mathcal G$ with
$t\in I^m$ as
$I_j=R_j\times I_j^m$, $1\leq j\leq N$, with
$\ell(I_1^m)\geq\cdots\geq\ell(I_N^m)$.
Set $E_j=R_j\setminus\bigcup_{i<j}R_i$.  The sets $E_j$ are
pairwise disjoint.  If $i<j$, the last cubes both contain $t$,
so $I_j^m\subset I_i^m$.  Therefore
\begin{align*}
 (R_j\setminus E_j)\times I_j^m
 &=\bigcup_{i<j}\bigl((R_j\cap R_i)\times I_j^m\bigr)\\
 &=\bigcup_{i<j}(I_j\cap I_i)
 \subset I_j\cap\sh(\mathcal G\setminus\{I_j\}).
\end{align*}
Taking measures and using \eqref{eq:selection},
\[
 (|R_j|-|E_j|)|I_j^m|
 \leq\tfrac12|I_j|
 =\tfrac12|R_j|\,|I_j^m|.
\]
Dividing by $|I_j^m|$ gives $|E_j|\geq|R_j|/2$.
Thus $\mathcal G_t$ is $1/2$-sparse.

Apply the sparse conclusion of Theorem~\ref{thm:weaker-containment}
with $m-1$ parameters to each slice. Since
\[
 \{x:(x,t)\in\sh(\mathcal G)\}=\sh(\mathcal G_t),
\]
Fubini's theorem and \eqref{eq:sliced-overlap} give, for
$c,C_2>0$ depending only on $m$ and $d$,
\begin{equation}\label{eq:selected-exponential}
\begin{aligned}
 \int_{\sh(\mathcal G)}e^{c h_{\mathcal G}^{1/(m-2)}}
 &=\int_{\R^{d_m}}\int_{\sh(\mathcal G_t)}
           e^{c h_{\mathcal G_t}(x)^{1/(m-2)}}\ud x\ud t\\
 &\leq C_2\int_{\R^{d_m}}|\sh(\mathcal G_t)|\ud t
 =C_2|\sh(\mathcal G)|.
\end{aligned}
\end{equation}

For these fixed constants there is $C_3$ such that
\begin{equation}\label{eq:young}
 C_1ab\leq C_3a\bigl(\log(e+a)\bigr)^{m-2}
          +\frac1{2C_2}e^{c b^{1/(m-2)}},\qquad a,b\geq0.
\end{equation}
To verify this, choose $K$ so large that $K^{-1/(m-2)}\leq c/2$ and
$2C_1C_2b\leq e^{(c/2)b^{1/(m-2)}}$ for $b>K$.
If $b\leq K(\log(e+a))^{m-2}$, then
$C_1ab\leq C_1K a(\log(e+a))^{m-2}$.
Otherwise $a\leq\exp((b/K)^{1/(m-2)})$ and $b>K$, so
\begin{align*}
 C_1ab&\leq C_1b\exp((b/K)^{1/(m-2)})
 \leq C_1b\exp\!\left(\tfrac c2 b^{1/(m-2)}\right)\\
 &\leq(2C_2)^{-1}\exp(c b^{1/(m-2)}).
\end{align*}
This proves \eqref{eq:young} with $C_3=C_1K$.
Apply it with $a=f$, $b=h_{\mathcal G}$ in
\eqref{eq:levelset-pairing}, and use \eqref{eq:selected-exponential}:
\[
\begin{aligned}
 |\Omega|
 &\leq C_3\int_{\R^d}f\bigl(\log(e+f)\bigr)^{m-2}
       +\frac1{2C_2}\int_{\sh(\mathcal G)}
                         e^{c h_{\mathcal G}^{1/(m-2)}}\\
 &\leq C_3\int_{\R^d}f\bigl(\log(e+f)\bigr)^{m-2}
       +\tfrac12|\sh(\mathcal G)|\\
 &\leq C_3\int_{\R^d}f\bigl(\log(e+f)\bigr)^{m-2}
       +\tfrac12|\Omega|.
\end{aligned}
\]
Moving the last term to the left gives
\[
 |\sh(\mathcal F)|
 \leq2C_3\int_{\R^d}f\bigl(\log(e+f)\bigr)^{m-2}.
\]
Choose increasing finite families $\mathcal F_N$ with
\[
 \bigcup_N\mathcal F_N
 =\{I\in\mathcal E_\phi^\tau:\langle f\rangle_I>1\}.
\]
Then
\[
\begin{aligned}
 \{M_{\mathcal E_\phi^\tau}f>1\}
 &=\bigcup_N\sh(\mathcal F_N),\\
 |\{M_{\mathcal E_\phi^\tau}f>1\}|
 &=\lim_{N\to\infty}|\sh(\mathcal F_N)|.
\end{aligned}
\]
This proves \eqref{eq:rounded-endpoint}.
\end{proof}

\begin{proof}[Proof of Theorem~\ref{thm:continuous-zygmund}]
There are $3^d$ choices of $\tau$. By
\eqref{eq:continuous-domination} and \eqref{eq:rounded-endpoint},
\begin{align*}
 |\{M_{\mathcal R_\phi}f>\lambda\}|
 &\leq\sum_\tau|\{M_{\mathcal E_\phi^\tau}f>\lambda/8^d\}|\\
 &\leq C_{m,d}\int_{\R^d}\frac{8^d|f|}{\lambda}
       \left[\log\!\left(e+\frac{8^d|f|}{\lambda}\right)\right]^{m-2}\\
 &\leq C_{m,d}\int_{\R^d}\frac{|f|}{\lambda}
       \left[\log\!\left(e+\frac{|f|}{\lambda}\right)\right]^{m-2}.
\end{align*}
The last step uses
$\log(e+8^du)\leq(1+d\log8)\log(e+u)$ for $u\geq0$.
This proves \eqref{eq:continuous-zygmund}.
\end{proof}

The dyadic version of \eqref{eq:zygmund} follows as a special case.
For a coordinatewise nondecreasing $\Phi\colon\Z^{m-1}\to\Z$, set
\[
 \phi(s_1,\ldots,s_{m-1})
 =2^{\Phi(\lfloor\log_2s_1\rfloor,\ldots,
                \lfloor\log_2s_{m-1}\rfloor)}.
\]
This function is positive and coordinatewise nondecreasing, and
$\phi(2^{k_1},\ldots,2^{k_{m-1}})=2^{\Phi(k_1,\ldots,k_{m-1})}$.
Thus $\mathcal Z_\Phi\subset\mathcal R_\phi$, so
$M_{\mathcal Z_\Phi}f\leq M_{\mathcal R_\phi}f$.

\section{Sparse \texorpdfstring{$\Phi$}{Phi}-Zygmund families}\label{sec:phi-sparse}

The Zygmund estimate also gives an overlap bound for sparse
$\Phi$-Zygmund families, without assuming incomparability.

\begin{corollary}\label{cor:phi-sparse}
Let $m\geq3$, let $\Phi\colon\Z^{m-1}\to\Z$ be coordinatewise
nondecreasing, and let $\mathcal G\subset\mathcal Z_\Phi$ be
$\eta$-sparse with finite shadow, where $0<\eta\leq1$. Then
\begin{equation}\label{eq:phi-sparse-exponential}
 \int_{\sh(\mathcal G)}
       \bigl[\exp(c h_{\mathcal G}^{1/(m-1)})-1\bigr]
 \leq C|\sh(\mathcal G)|,
\end{equation}
where $c,C>0$ depend only on $m$, $d$, and $\eta$, not on $\Phi$
or the dyadic grids.
\end{corollary}

\begin{proof}
The last paragraph of Section~\ref{sec:continuous-zygmund} and
Theorem~\ref{thm:continuous-zygmund} give the weak
$L(\log L)^{m-2}$ estimate for $M_{\mathcal Z_\Phi}$.
The maximal-function-to-overlap implication in
Rey~\cite[Theorem~1.1(a)]{Rey}, with $k=m-2$, gives
\eqref{eq:phi-sparse-exponential}.
\end{proof}

Thus sparseness with the $\Phi$-Zygmund scale restriction gives the
same exponential power as sparseness with incomparability in
Theorem~\ref{thm:overlap}. Adding incomparability to a sparse
$\Phi$-Zygmund family need not improve this power, even for the
classical scale relation.

\begin{theorem}[Sharpness under incomparability]\label{thm:phi-sparse-sharpness}
Let $m\geq3$, $0<\eta<1$, and set
\[
 \Phi(k_1,\ldots,k_{m-1})=k_1+\cdots+k_{m-1}.
\]
There are finite incomparable $\eta$-sparse families
$\mathcal G_N\subset\mathcal Z_\Phi$, $N\geq2$, with
$|\sh(\mathcal G_N)|=1$, such that
\[
 \lim_{N\to\infty}
 \int_{\sh(\mathcal G_N)}
       \bigl[\exp(c h_{\mathcal G_N}^{\beta})-1\bigr]
 =\infty
 \qquad\text{for every }c>0,\quad\beta>\frac1{m-1}.
\]
\end{theorem}

\begin{proof}
Fix $N\geq2$. We choose $N^{m-1}$ side-length tuples so that, between
any two distinct tuples, one side length is larger and another is smaller.
Rectangles with different tuples therefore cannot contain one another.
For each tuple, we choose a disjoint collection of rectangles, so their
union is incomparable. It will also be sparse and have overlap $N^{m-1}$
on a set of measure at least $e^{-CN}$, where $C$ is independent of $N$.
This suffices for the failure of every larger exponential power, since
for $c>0$ and $\beta>1/(m-1)$,
\[
 e^{-CN}\bigl[\exp(cN^{(m-1)\beta})-1\bigr]
 =\exp\bigl(cN^{(m-1)\beta}-CN\bigr)-e^{-CN}
 \longrightarrow\infty
 \qquad(N\to\infty).
\]

Choose unit cubes $Q^i\in\cD^i$, and put
$\Omega=Q^1\times\cdots\times Q^m$. Fix an integer $s\geq1$ such that
\[
 (1-2^{-s})^{m-1}\geq\eta.
\]
Set $\delta=2^{-s}$ and $L=(2N+m-2)(N-1)$.

\emph{The side lengths.}
For $a=(a_2,\ldots,a_m)\in\{0,\ldots,N-1\}^{m-1}$, we use side lengths
$2^{-n_i(a)}$, where
\begin{equation}\label{eq:sharp-generations}
\begin{aligned}
 n_1(a)&=s\left(L-2Na_2-\sum_{i=3}^{m}a_i\right),&
 n_2(a)&=s(2Na_2+2a_m),\\
 n_i(a)&=sa_i\quad(3\leq i\leq m-1),&
 n_m(a)&=s(L+a_m).
\end{aligned}
\end{equation}
These exponents enforce the Zygmund scale relation and incomparability;
the separation between successive $a_2$-levels also accommodates the
variation in $a_m$ in the sparse construction below.
All $n_i(a)$ are nonnegative, since
\[
 2Na_2+\sum_{i=3}^{m}a_i
 \leq 2N(N-1)+(m-2)(N-1)=L.
\]
Moreover,
\[
\begin{aligned}
 \sum_{i=1}^{m-1}n_i(a)
 &=s\left(L-2Na_2-\sum_{i=3}^{m}a_i
       +2Na_2+2a_m+\sum_{i=3}^{m-1}a_i\right)\\
 &=s(L+a_m)=n_m(a).
\end{aligned}
\]
Consequently,
\[
 2^{-n_m(a)}=\prod_{i=1}^{m-1}2^{-n_i(a)}.
\]
Thus dyadic rectangles with these side lengths belong to $\mathcal Z_\Phi$.

Suppose $n_i(a)\geq n_i(a')$ for every $i$. Coordinates
$3,\ldots,m$ give $a_i\geq a'_i$ for $3\leq i\leq m$.
If $a_2<a'_2$, then
\[
 \frac{n_2(a)-n_2(a')}{s}
 =2N(a_2-a'_2)+2(a_m-a'_m)
 \leq-2N+2(N-1)=-2<0,
\]
contrary to $n_2(a)\geq n_2(a')$. Thus $a_2\geq a'_2$.
The first coordinate then forces
\begin{equation}\label{eq:sharp-incomparability}
 0\leq\frac{n_1(a)-n_1(a')}{s}
 =-2N(a_2-a'_2)-\sum_{i=3}^{m}(a_i-a'_i)\leq0.
\end{equation}
Hence
\[
 2N(a_2-a'_2)+\sum_{i=3}^{m}(a_i-a'_i)=0.
\]
Since $a_i-a'_i\geq0$ for every $2\leq i\leq m$, each of these
differences must vanish. Thus $a=a'$. Rectangles with
different indices therefore cannot contain one another.

\emph{The nested sets.}
For each coordinate $2\leq i\leq m$, set $A^i_0=Q^i$.
Define, for $0\leq k\leq N$,
\[
 b_2(k)=2Nk,\qquad
 b_i(k)=k\quad(3\leq i\leq m-1),\qquad
 b_m(k)=L+k.
\]
Holding $i$ fixed, for $k=1,\ldots,N$, suppose inductively that
$A^i_0,\ldots,A^i_{k-1}$ have been constructed, with
\[
 Q^i=A^i_0\supset A^i_1\supset\cdots\supset A^i_{k-1},
\]
$|A^i_j|=\delta^j$ and each $A^i_j$ a union of dyadic cubes of side
length $2^{-sb_i(j)}$ for $0\leq j<k$.
These properties hold for $A^i_0=Q^i$, since $b_i(0)\geq0$.
Since
\[
 b_i(k)-1-b_i(k-1)
 =\begin{cases}
 2N-1,&i=2,\\
 0,&3\leq i\leq m,
 \end{cases}
 \geq0,
\]
we can partition $A^i_{k-1}$
into dyadic cubes $P$ with
$\ell(P)=2^{-s(b_i(k)-1)}$.
In each $P$, retain any $2^{s(d_i-1)}$ of its $2^{sd_i}$ dyadic
subcubes of side length $2^{-sb_i(k)}$, and let $A^i_k$ be the union of
all the retained cubes. Each such $P$ therefore satisfies
\[
 |P\cap A^i_k|
 =2^{s(d_i-1)}2^{-sd_i}|P|=\delta|P|.
\]
Since the cubes $P$ partition $A^i_{k-1}$,
\[
 |A^i_k|
 =\sum_P|P\cap A^i_k|
 =\delta\sum_P|P|
 =\delta|A^i_{k-1}|
 =\delta\delta^{k-1}
 =\delta^k.
\]
Thus $A^i_k$ is a union of dyadic cubes of side length
$2^{-sb_i(k)}$, and
\begin{equation}\label{eq:sharp-nested-measures}
 A^i_k\subset A^i_{k-1},\qquad |A^i_k|=\delta^k,
 \qquad 1\leq k\leq N.
\end{equation}

The scales of these sets and the chosen rectangle side lengths satisfy
\begin{equation}\label{eq:sharp-selection-scales}
 b_i(a_i)\leq n_i(a)/s\leq b_i(a_i+1)-1,
 \qquad 2\leq i\leq m.
\end{equation}
To see \eqref{eq:sharp-selection-scales}, notice that for $i\geq3$,
\[
 b_i(a_i)=\frac{n_i(a)}s=b_i(a_i+1)-1
 =\begin{cases}
 a_i,&3\leq i\leq m-1,\\
 L+a_m,&i=m,
 \end{cases}
\]
whereas for $i=2$,
\[
\begin{aligned}
 b_2(a_2)=2Na_2
 &\leq\frac{n_2(a)}s=2Na_2+2a_m\\
 &\leq2Na_2+2(N-1)
 \leq2N(a_2+1)-1=b_2(a_2+1)-1.
\end{aligned}
\]

Define
\[
\begin{aligned}
 \mathcal G(a)&=\left\{I\in\cD:
 \begin{gathered}
 I\subset\Omega,\quad \ell(I^i)=2^{-n_i(a)}\quad(1\leq i\leq m),\\
 I^i\subset A^i_{a_i}\quad(2\leq i\leq m)
 \end{gathered}\right\},\\
 \mathcal G_N&=\bigcup_a\mathcal G(a),
\end{aligned}
\]
where $a$ runs over $\{0,\ldots,N-1\}^{m-1}$.
The lower bound in
\eqref{eq:sharp-selection-scales} shows that $\mathcal G(a)$ partitions
\[
 Q^1\times\prod_{i=2}^m A^i_{a_i}.
\]
Since $\mathcal G(0)$
partitions $\Omega$,
\[
 \sh(\mathcal G_N)=\Omega,\qquad |\sh(\mathcal G_N)|=1.
\]
Rectangles within one $\mathcal G(a)$ are disjoint. Rectangles from
different indices cannot contain one another by
\eqref{eq:sharp-incomparability}. Hence $\mathcal G_N$ is incomparable.

\emph{Sparseness and overlap.}
For $I\in\mathcal G(a)$, define
\[
 E_I=I^1\times\prod_{i=2}^m
                  \bigl(I^i\setminus A^i_{a_i+1}\bigr).
\]
Since $I^i\subset A^i_{a_i}$ for $2\leq i\leq m$,
\[
 E_I\subset Q^1\times\prod_{i=2}^m
                  \bigl(A^i_{a_i}\setminus A^i_{a_i+1}\bigr).
\]
For $2\leq i\leq m$, the construction at step $k=a_i+1$ partitions
$A^i_{a_i}$ into dyadic cubes $P$ with
\[
 \ell(P)=2^{-s(b_i(a_i+1)-1)}
 \leq2^{-n_i(a)}=\ell(I^i),
\]
where the inequality follows from the upper bound in
\eqref{eq:sharp-selection-scales}. Since $I^i\subset A^i_{a_i}$,
the cubes $P\subset I^i$ partition $I^i$. Hence
\[
 |I^i\cap A^i_{a_i+1}|
 =\sum_{P\subset I^i}|P\cap A^i_{a_i+1}|
 =\delta\sum_{P\subset I^i}|P|
 =\delta|I^i|,
\]
and therefore
\[
\begin{aligned}
 |E_I|
 &=|I^1|\prod_{i=2}^m\bigl(|I^i|-|I^i\cap A^i_{a_i+1}|\bigr)\\
 &=|I^1|\prod_{i=2}^m(1-\delta)|I^i|
 =(1-\delta)^{m-1}|I|\geq\eta|I|.
\end{aligned}
\]
For a fixed $a$, the sets $E_I$ are disjoint because the rectangles
in $\mathcal G(a)$ are disjoint. For distinct $a,a'$, after interchanging
them if necessary, choose a coordinate $i$ with $a_i<a'_i$.
Nesting gives $A^i_{a'_i}\subset A^i_{a_i+1}$, so
\[
 (A^i_{a_i+1})^c\cap A^i_{a'_i}=\varnothing.
\]
Thus all the $E_I$ are pairwise disjoint, proving $\eta$-sparseness.

The partitions defining $\mathcal G(a)$ give the exact overlap
\begin{equation}\label{eq:sharp-overlap-product}
\begin{aligned}
 h_{\mathcal G_N}(x)
 &=\1_{Q^1}(x_1)
     \sum_{a\in\{0,\ldots,N-1\}^{m-1}}
                  \prod_{i=2}^m\1_{A^i_{a_i}}(x_i)\\
 &=\1_{Q^1}(x_1)\prod_{i=2}^m
                  \left(\sum_{k=0}^{N-1}\1_{A^i_k}(x_i)\right).
\end{aligned}
\end{equation}
Each factor is at most $N$ and equals $N$ precisely on $A^i_{N-1}$.
Consequently, using \eqref{eq:sharp-nested-measures},
\[
 \{h_{\mathcal G_N}=N^{m-1}\}
 =Q^1\times\prod_{i=2}^m A^i_{N-1}.
\]
With $C=(m-1)\log(1/\delta)=(m-1)s\log2$, independent of $N$,
\[
 |\{h_{\mathcal G_N}=N^{m-1}\}|
 =\delta^{(m-1)(N-1)}
 =e^{-C(N-1)}\geq e^{-CN}.
\]
The calculation at the beginning of the proof now gives the asserted
divergence of the exponential integrals.
\end{proof}

\bibliographystyle{amsplain}
\bibliography{references}
\end{document}